\documentclass[11pt]{article}

\usepackage[nottoc,notlot,notlof]{tocbibind}

\usepackage{amsmath,amssymb,amsfonts,amsthm}
\usepackage{latexsym}
\usepackage{tikz}
\usetikzlibrary{automata,positioning}
\usetikzlibrary{chains,fit,shapes}
\usetikzlibrary{calc}
\usetikzlibrary{arrows}
\usepackage{tkz-graph}
\usetikzlibrary{positioning,calc}
\usetikzlibrary{graphs}
\usetikzlibrary{graphs.standard}
\usetikzlibrary{arrows,decorations.markings}
\usepackage{multicol}
\usepackage{xcolor}
\newtheorem{theorem}{Theorem}
\newtheorem{definition}[theorem]{Definition}
\newtheorem{lemma}[theorem]{Lemma}
\newtheorem{remark}[theorem]{Remark}
\newtheorem{proposition}[theorem]{Proposition}
\newtheorem{corollary}[theorem]{Corollary}

\newtheorem{example}[theorem]{Example}
\newcommand{\IWR}{\rm IWR}

\title{Word-representability of split graphs generated by morphisms}

\author{Kittitat Iamthong\footnote{Department of Mathematics and Statistics, University of Strathclyde, 26 Richmond Street Glasgow G1 1XH, United Kingdom. 
{\bf Email:} kittitat.iamthong@strath.ac.uk .}}

\begin{document}  

\maketitle

\abstract{ A graph $G=(V,E)$ is word-representable if and only if there exists a word $w$ over the alphabet $V$ such that letters $x$ and $y$, $x\neq y$, alternate in $w$ if and only if $xy\in E$. A split graph is a graph in which the vertices can be partitioned into a clique and an independent set. There is a long line of research on word-representable graphs in the literature, and recently, word-representability of split graphs has attracted interest. 

In this paper, we first give a characterization of word-representable split graphs in terms of permutations of columns of the adjacency matrices. Then, we focus on the study of word-representability of split graphs obtained by iterations of a morphism, the notion coming from combinatorics on words. We prove a number of general theorems and provide a complete classification in the case of morphisms defined by $2\times 2$ matrices.}

\section{Introduction}\label{Sec-intro}

A graph $G=(V,E)$ is {\em word-representable} if and only if there exists a word $w$ over the alphabet $V$ such that letters $x$ and $y$, $x\neq y$, alternate in $w$ if and only if $xy\in E$. In this definition, letters $x$ and $y$ {\em alternate} in $w$ if after removing all letters but $x$ and $y$, we would either get a word of the form $xyxy\cdots$, or of the form $yxyx\cdots$, of even or odd length. Also, by definition, $w$ must contain each letter in $V$ at least once. For example, the cycle graph on four vertices labeled 1, 2, 3, 4 in clockwise direction is word-representable because it can be represented by the word 14213243.

There is a long line of research on word-representable graphs in the literature that is summarized in \cite{K17}. Word-representable graphs are important as they generalize several well-known and well-studied classes of graphs such as 3-colorable graphs, comparability graphs and circle graphs \cite{KL15}. Note that the class of word-representable graphs is hereditary, that is, removing a vertex in a word-representable graph results in a word-representable graph. The wheel graph $W_5$ is the smallest non-word-representable graph. One of the key results in the area is the following theorem, where an orientation of a graph is {\it semi-transitive} if it is acyclic, and for any directed path $v_0\rightarrow v_1\rightarrow \cdots\rightarrow v_k$ either there is no edge between $v_0$ and $v_k$, or $v_i\rightarrow v_j$ is an edge for all $0\leq i<j\leq k$.

\begin{theorem}[\cite{HKP16}]\label{key-thm} A graph is word-representable if and only if it admits a semi-transitive orientation. \end{theorem}

The following simple lemma will also be of use to us in this paper.
\begin{lemma}[\cite{KLMW17}]\label{lem-tran-orie} Let $K_m$ be a clique in a graph $G$. Then any acyclic orientation of $G$ induces a transitive orientation on $K_m$ (where the presence of edges $u\rightarrow v$ and $v\rightarrow z$ implies the presence of the edge $u\rightarrow z$). In particular, any semi-transitive orientation of $G$ induces a transitive orientation on $K_m$. In either case, the orientation induced on $K_m$ contains a single source and a single sink. \end{lemma}

\subsection{Split graphs.} A {\em split graph} is a graph in which the vertices can be partitioned into a clique and an independent set \cite{FH77}. The paper \cite{KLMW17} initiated a systematic study of word-representability of split graphs, which was extended in a follow up paper \cite{CKS19}.  In particular, characterizations of split graphs in terms of forbidden induced subgraphs were obtained in \cite{KLMW17} and  \cite{CKS19} for cliques of sizes 4 and 5, respectively. Also, a characterization of semi-transitive orientations of split graphs was obtained in \cite{KLMW17} (see below), and split graphs were used to solve a long standing problem in the theory of word-representation in \cite{CKS19}. We note though that currently a complete characterization of split graphs (e.g. in terms of forbidden subgraphs) seems to be a non-feasible problem, so a natural research direction is in understanding (non-)word-representable subclasses of split graphs.

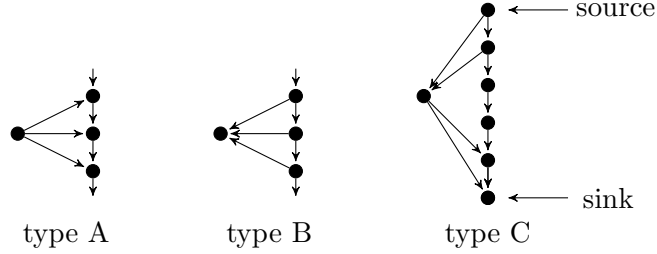
\begin{figure}
	\begin{center}
		
		\begin{tabular}{ccc}
			
			\begin{tikzpicture}[->,>=stealth', shorten >=1pt,node distance=0.5cm,auto,main node/.style={fill,circle,draw,inner sep=0pt,minimum size=5pt}]
			
			\node[main node] (1) {};
			\node[main node] (2) [right of=1,xshift=5mm] {};
			\node[main node] (3) [above of=2] {};
			\node (4) [above of=3] {};
			\node[main node] (5) [below of=2] {};
			\node (6) [below of=5] {};
			
			\node (7) [above of=4,yshift=-4mm,xshift=1mm] {};
			\node (8) [above right of=7,xshift=6mm]{};
			
			\node (10) [below of=6,yshift=4mm,xshift=1mm] {};
			\node (11) [below right of=10,xshift=6mm]{};
			
			\node (11) [below left of=5,yshift=-5mm]{type A};
			
			\path
			(4) edge (3)
			(3) edge (2)
			(2) edge (5)
			(5) edge (6);
			
			\path
			(1) edge (2)
			(1) edge (3)
			(1) edge (5);
			
			\end{tikzpicture}
			
			&
			\begin{tikzpicture}[->,>=stealth', shorten >=1pt,node distance=0.5cm,auto,main node/.style={fill,circle,draw,inner sep=0pt,minimum size=5pt}]
			
			\node[main node] (1) {};
			\node[main node] (2) [right of=1,xshift=5mm] {};
			\node[main node] (3) [above of=2] {};
			\node (4) [above of=3] {};
			\node[main node] (5) [below of=2] {};
			\node (6) [below of=5] {};
			
			\node (7) [above of=4,yshift=-4mm,xshift=1mm] {};
			\node (8) [above right of=7,xshift=6mm]{};
			
			\node (10) [below of=6,yshift=4mm,xshift=1mm] {};
			\node (11) [below right of=10,xshift=6mm]{};
			
			\node (11) [below left of=5,yshift=-5mm]{type B};
			
			\path
			(4) edge (3)
			(3) edge (2)
			(2) edge (5)
			(5) edge (6);
			
			\path
			(2) edge (1)
			(3) edge (1)
			(5) edge (1);
			
			\end{tikzpicture}
			
			&
			
			\begin{tikzpicture}[->,>=stealth', shorten >=1pt,node distance=0.5cm,auto,main node/.style={fill,circle,draw,inner sep=0pt,minimum size=5pt}]
			
			\node[main node] (1) {};
			\node[main node] (2) [below right of=1,xshift=5mm] {};
			\node[main node] (3) [above of=2] {};
			\node[main node] (4) [above of=3] {};
			\node[main node] (5) [below of=2] {};
			\node[main node] (6) [below of=5] {};
			\node[main node] (13) [above of=4] {};
			\node[main node] (14) [below of=5] {};
			
			\node (7) [above of=13,yshift=-5mm,xshift=1mm] {};
			\node (8) [right of=7,xshift=6mm]{};
			\node (9) [right of=8]{source};
			
			\node (10) [below of=14,yshift=5mm,xshift=1mm] {};
			\node (11) [right of=10,xshift=6mm]{};
			\node (12) [right of=11,xshift=-1mm]{sink};
			
			\node [below of=14]{type C};
			
			\path
			(8) edge (7)
			(11) edge (10);
			
			\path
			(13) edge (1)
			(4) edge (1)
			(1) edge (14)
			(1) edge (5)
			(4) edge (3)
			(3) edge (2)
			(2) edge (5)
			(5) edge (6)
			(13) edge (4)
			(5) edge (14);
			
			\end{tikzpicture}
			
		\end{tabular}
		
\caption{\label{3-groups} Three types of vertices in $E_{n-m}$ in a semi-transitive orientation of $(E_{n-m},K_m)$. The vertical oriented paths are a schematic way to show (parts of) $\vec{P}$}
	\end{center}
\end{figure}

Let $S_n=(E_{n-m},K_m)$ be a word-representable split graph, where $K_m$ is  the maximal clique, and $E_{n-m}$ is the independent set.  Then, by Theorem~\ref{key-thm}, $S_n$ admits a semi-transitive orientation. Further, by Lemma~\ref{lem-tran-orie} we know that any such orientation induces a transitive orientation on $K_m$ with the longest directed path $\vec{P}$. Theorem~\ref{main-orientation} below characterizes semi-transitive orientations of split graphs. 

\begin{theorem}[\cite{KLMW17}]\label{semi-tran-groups} Any semi-transitive orientation of a split graph $S_n=(E_{n-m},K_m)$ subdivides the set of all vertices in $E_{n-m}$ into three, possibly empty, groups corresponding to each of the following types (also shown schematically in Figure~\ref{3-groups}), where $\vec{P}= p_1\rightarrow\cdots\rightarrow p_m$ is the longest directed path in $K_m$: 
\begin{itemize}
\item A vertex in $E_{n-m}$ is of {\em type A} if it is a source and is connected to all vertices in $\{p_i,p_{i+1},\ldots, p_j\}$ for some $1\leq i\leq j\leq m$;
\item A vertex in $E_{n-m}$ is of {\em type B} if it is a sink and is connected to all vertices in $\{p_i,p_{i+1},\ldots, p_j\}$ for some $1\leq i\leq j\leq m$; 
\item A vertex $v\in E_{n-m}$ is of {\em type C} if there is an edge $x\rightarrow v$ for each $x\in I_v=\{p_1,p_2,\ldots, p_i\}$ and there is an edge $v\rightarrow y$ for each $y\in O_v=\{p_j,p_{j+1},\ldots, p_m\}$ for some $1\leq i< j\leq m$.
\end{itemize} 
\end{theorem}

There are additional restrictions, given by the next theorem, on relative positions of the neighbours of vertices of types A, B and C. 

\begin{theorem}[\cite{KLMW17}]\label{relative-order} Let $S_n=(E_{n-m},K_m)$ be oriented semi-transitively with $\vec{P}= p_1\rightarrow\cdots\rightarrow p_m$. For a vertex $x\in E_{n-m}$ of type C,
there is no vertex $y\in E_{n-m}$ of type A or B, which is connected to both $p_{|I_x|}$ and $p_{m-|O_x|+1}$. Also, there is no vertex $y\in E_{n-m}$ of type C such that either $I_y$, or $O_y$ contains both $p_{|I_x|}$ and $p_{m-|O_x|+1}$.
\end{theorem}

One can now classify semi-transitive orientations on split graphs.

\begin{theorem}[\cite{KLMW17}]\label{main-orientation} An orientation of a split graph $S_n=(E_{n-m},K_m)$ is semi-transitive if and only if 
\begin{itemize} 
\item[{\em (i)}] $K_m$ is oriented transitively;
\item[{\em (ii)}] each vertex in $E_{n-m}$ is of one of the three types in Theorem~\ref{semi-tran-groups};
\item[{\em (iii)}] the restrictions in Theorem~\ref{relative-order} are satisfied. 
\end{itemize}\end{theorem}

\subsection{Split graphs generated by morphisms} Let $A$ and $B$ be alphabets (possibly $A=B$). A map $\varphi: A^* \rightarrow B^*$ is called a {\em morphism}, if
we have $\varphi(uv)=\varphi(u)\varphi(v)$ for any $u,v \in A^*$. A
morphism $\varphi$ can be defined by defining $\varphi(a)$ for each
$a \in A$. A particular property of a morphism $\varphi$ is that
$\varphi(\varepsilon)=\varepsilon$, where $\varepsilon$ is the empty word. Morphisms are a central object in the area of combinatorics on words \cite{Loth}, and there is a natural extension of the notion to two, or more, dimensions. Indeed, one can begin with a matrix $M$ whose entries are elements of $A$, and then obtain $\varphi(M)$ by substituting each element in $M$ by matrices having the same dimensions and given by some substitution rules.

Relevance of (2-dimensional) morphisms to split graphs is coming through the ideas communicated in \cite{WRofToeplitz}, where patterns in adjacency matrices are considered to study word-representability of graphs, and the notion of an infinite word-representable graph is introduced. In this paper, we study word-representability of families of split graphs defined by iteration of morphisms, and in particular, we give a complete classification in the case of $2\times 2$ matrices. However, a key theorem we prove that characterizes word-representability of split graphs in terms of permutations of columns in the adjacency matrix is applicable to {\em any} split graph (see Theorem~\ref{Thm-genpermutecolumn}). 

\subsection{Organization of the paper} After introducing a number of preliminary results and examples in Section~\ref{prelim-sec}, that includes our key result (Theorem~\ref{Thm-genpermutecolumn}),  in Section~\ref{general-sec} we present a number of general results on split graphs generated by morphisms. In Section~\ref{class-2x2-sec} we provide a complete classification of word-representable split graphs defined by iteration of morphisms using two $2\times 2$ matrices; our results in this section are summarized in Tables~\ref{tab:2x2-1} and~\ref{tab:2x2-2}. Finally, in Section~\ref{final-sec} we state a number of open research questions.

\begin{center}
	\begin{table}
	{\tiny
	\begin{tabular}{|c|c|c|c|c||c|c|c|c|c|} 
		\hline
		Case&$A$&$B$&$\IWR$$(A,B)$&Ref.&Case&$A$&$B$&$\IWR$$(A,B)$& Ref.\\
		\hline\hline
		1&$\begin{bmatrix} 0&0\\0&0\end{bmatrix}$ &
		$\begin{bmatrix} *&*\\ *&*\end{bmatrix}$& $\infty$&  $A=0$ &
		&&&&  \\
		\hline \hline
		2&$\begin{bmatrix} 1&0\\0&0\end{bmatrix}$ &
		$\begin{bmatrix} 0&0\\ 0&0\end{bmatrix}$& $\infty$ &Prop.~\ref{ABall0all1}  &
		18&$\begin{bmatrix} 0&1\\0&0\end{bmatrix}$ &
		$\begin{bmatrix} 0&0\\0&0\end{bmatrix}$& $\infty$& Prop.~\ref{ABall0all1} \\
		\hline
		3&&$\begin{bmatrix} 1&0\\ 0&0\end{bmatrix}$& $\infty$ &  Cor.~\ref{A=B22} &
		19&&$\begin{bmatrix} 1&0\\0&0\end{bmatrix}$& 3 & Rem.~\ref{REMindex=3} \\
		\hline
		4&&$\begin{bmatrix} 0&1\\ 0&0\end{bmatrix}$& 3 & Pro.~\ref{caseindex=3} &
		20&&$\begin{bmatrix} 0&1\\0&0\end{bmatrix}$& $\infty$&  Cor.~\ref{A=B22}\\
		\hline
		5&&$\begin{bmatrix} 0&0\\ 1&0\end{bmatrix}$& 3 & Rem.~\ref{REMindex=3} &
		21&&$\begin{bmatrix} 0&0\\1&0\end{bmatrix}$& 4 & Rem.~\ref{REMindex=4} \\
		\hline		
		6&&$\begin{bmatrix} 0&0\\ 0&1\end{bmatrix}$& 4 & Prop.~\ref{caseindex=4} &
		22&&$\begin{bmatrix} 0&0\\0&1\end{bmatrix}$& 3& Rem.~\ref{REMindex=3} \\
		\hline		
		7&&$\begin{bmatrix} 1&1\\ 0&0\end{bmatrix}$& $\infty$ & Rem~\ref{Remindex=inf} &
		23&&$\begin{bmatrix} 1&1\\0&0\end{bmatrix}$& $\infty$ & Rem~\ref{Remindex=inf} \\
		\hline
		8&&$\begin{bmatrix} 1&0\\ 1&0\end{bmatrix}$& $\infty$ & Rem~\ref{Remindex=inf}  &
		24&&$\begin{bmatrix} 1&0\\1&0\end{bmatrix}$& 3 &  Rem.~\ref{REMindex=3} \\
		\hline
		9&&$\begin{bmatrix} 1&0\\ 0&1\end{bmatrix}$& $\infty$ & Prop~\ref{case9}   &
		25&&$\begin{bmatrix} 1&0\\0&1\end{bmatrix}$& 3 & Rem.~\ref{REMindex=3} \\
		\hline
		10&&$\begin{bmatrix} 0&1\\ 1&0\end{bmatrix}$& 3 &Rem.~\ref{REMindex=3} &
		26&&$\begin{bmatrix} 0&1\\1&0\end{bmatrix}$& $\infty$ & Rem~\ref{case9-26-41} \\
		\hline
		11&&$\begin{bmatrix} 0&1\\ 0&1\end{bmatrix}$& 3 & Rem.~\ref{REMindex=3} &
		27&&$\begin{bmatrix} 0&1\\0&1 \end{bmatrix}$& $\infty$& Rem~\ref{Remindex=inf} \\
		\hline
		12&&$\begin{bmatrix} 0&0\\ 1&1\end{bmatrix}$& 3 & Rem.~\ref{REMindex=3} &
		28&&$\begin{bmatrix} 0&0\\1&1\end{bmatrix}$& 3 & Rem.~\ref{REMindex=3} \\
		\hline
		13&&$\begin{bmatrix} 1&1\\ 1&0\end{bmatrix}$& $\infty$ & Rem~\ref{Remindex=inf}  &
		29&&$\begin{bmatrix} 1&1\\1&0\end{bmatrix}$& 3 & Rem.~\ref{REMindex=3} \\
		\hline
		14&&$\begin{bmatrix} 1&1\\ 0&1\end{bmatrix}$& 3 & Rem.~\ref{REMindex=3} &
		30&&$\begin{bmatrix} 1&1\\0&1\end{bmatrix}$& $\infty$ & Rem~\ref{Remindex=inf} \\
		\hline
		15&&$\begin{bmatrix} 1&0\\ 1&1\end{bmatrix}$& 3 & Rem.~\ref{REMindex=3}&
		31&&$\begin{bmatrix} 1&0\\1&1\end{bmatrix}$& 2 & Rem.~\ref{REMindex=2} \\ %case 16 
		\hline
		16&&$\begin{bmatrix} 0&1\\ 1&1\end{bmatrix}$& 2 & Prop.~\ref{caseindex=2} &
		32&&$\begin{bmatrix} 0&1\\1&1\end{bmatrix}$& 3 & Rem.~\ref{REMindex=3}\\
		\hline
		17&&$\begin{bmatrix} 1&1\\ 1&1\end{bmatrix}$& 3 &Rem.~\ref{REMindex=3} &
		33&&$\begin{bmatrix} 1&1\\1&1\end{bmatrix}$& 3 & Rem.~\ref{REMindex=3} \\
		\hline\hline
		34&$\begin{bmatrix} 0&0\\0&1\end{bmatrix}$ &
		$\begin{bmatrix} 0&0\\ 0&0\end{bmatrix}$& $\infty$& Prop.~\ref{ABall0all1} &
		50&$\begin{bmatrix} 1&1\\0&0\end{bmatrix}$ &
		$\begin{bmatrix} 0&0\\0&0\end{bmatrix}$& $\infty $&  Prop.~\ref{ABall0all1} \\
		\hline
		35&&$\begin{bmatrix} 1&0\\ 0&0\end{bmatrix}$& 4&  Rem.~\ref{REMindex=4} &
		51&&$\begin{bmatrix} 1&0\\0&0\end{bmatrix}$& $\infty$& Rem~\ref{Remindex=inf} \\
		\hline
		36&&$\begin{bmatrix} 0&1\\ 0&0\end{bmatrix}$& 3 & Rem.~\ref{REMindex=3} &
		52&&$\begin{bmatrix} 0&1\\0&0\end{bmatrix}$& $\infty$& Rem~\ref{Remindex=inf} \\
		\hline
		37&&$\begin{bmatrix} 0&0\\ 1&0\end{bmatrix}$& 3 & Rem.~\ref{REMindex=3} &
		53&&$\begin{bmatrix} 0&0\\1&0\end{bmatrix}$& 4 &Rem.~\ref{REMindex=4}   \\
		\hline		
		38&&$\begin{bmatrix} 0&0\\ 0&1\end{bmatrix}$& $\infty$ &  Cor.~\ref{A=B22} &
		54&&$\begin{bmatrix} 0&0\\0&1\end{bmatrix}$& 4 &  Rem.~\ref{REMindex=4} \\
		\hline		
		39&&$\begin{bmatrix} 1&1\\ 0&0\end{bmatrix}$& 3 &  Rem.~\ref{REMindex=3} &
		55&&$\begin{bmatrix} 1&1\\0&0\end{bmatrix}$ & $\infty$ &   Cor.~\ref{A=B22}\\
		\hline
		40&&$\begin{bmatrix} 1&0\\ 1&0\end{bmatrix}$& 3&  Rem.~\ref{REMindex=3} &
		56&&$\begin{bmatrix} 1&0\\1&0\end{bmatrix}$& $\infty$ & Rem~\ref{Remindex=inf}  \\
		\hline
		41&&$\begin{bmatrix} 1&0\\ 0&1\end{bmatrix}$& $\infty$ & Rem~\ref{case9-26-41} &
		57&&$\begin{bmatrix} 1&0\\0&1\end{bmatrix}$& 3 &Rem.~\ref{REMindex=3} \\
		\hline
		42&&$\begin{bmatrix} 0&1\\ 1&0\end{bmatrix}$& 3 &  Rem.~\ref{REMindex=3} &
		58&&$\begin{bmatrix} 0&1\\1&0\end{bmatrix}$& 3 & Rem.~\ref{REMindex=3}\\
		\hline
		43&&$\begin{bmatrix} 0&1\\ 0&1\end{bmatrix}$& $\infty$ & Rem~\ref{Remindex=inf}  &
		59&&$\begin{bmatrix} 0&1\\0&1 \end{bmatrix}$& $\infty$ & Rem~\ref{Remindex=inf} \\
		\hline
		44&&$\begin{bmatrix} 0&0\\ 1&1\end{bmatrix}$& $\infty$ & Rem~\ref{Remindex=inf}  &
		60&&$\begin{bmatrix} 0&0\\1&1\end{bmatrix}$& $\infty$ &  Prop. \ref{ABall0all1} \\
		\hline
		45&&$\begin{bmatrix} 1&1\\ 1&0\end{bmatrix}$& 2 & Rem.~\ref{REMindex=2} &
		% case 16
		61&&$\begin{bmatrix} 1&1\\1&0\end{bmatrix}$& $\infty$ & Rem~\ref{Remindex=inf}  \\
		\hline
		46&&$\begin{bmatrix} 1&1\\ 0&1\end{bmatrix}$& 3 & Rem.~\ref{REMindex=3} &
		62&&$\begin{bmatrix} 1&1\\0&1\end{bmatrix}$& $\infty$ & Rem~\ref{Remindex=inf} \\
		\hline
		47&&$\begin{bmatrix} 1&0\\ 1&1\end{bmatrix}$& 3 &  Rem.~\ref{REMindex=3}  &
		63&&$\begin{bmatrix} 1&0\\1&1\end{bmatrix}$& 4 &Rem.~\ref{REMindex=4} \\
		\hline
		48&&$\begin{bmatrix} 0&1\\ 1&1\end{bmatrix}$& $\infty$ & Rem~\ref{Remindex=inf} &
		64&&$\begin{bmatrix} 0&1\\1&1\end{bmatrix}$& 4 & Rem.~\ref{REMindex=4} \\
		\hline
		49&&$\begin{bmatrix} 1&1\\ 1&1\end{bmatrix}$& 3 & Rem.~\ref{REMindex=3} &
		65&&$\begin{bmatrix} 1&1\\1&1\end{bmatrix}$& $\infty$ &  Prop. \ref{ABall0all1} \\
		\hline
			\end{tabular} }
		\caption{The index of word-representability of infinite split graphs $G(A,B)$ for $2 \times 2$ matrices $A$ and $B$.}\label{tab:2x2-1} 
		\end{table}

\end{center}

\begin{center}
	\begin{table}
	{\tiny
		\begin{tabular}{|c|c|c|c|c||c|c|c|c|c|}
			\hline
			Case&A&B&$\IWR$$(A,B)$&Ref.&Case&
			A&B&$\IWR$$(A,B)$&Ref.\\
			\hline\hline
			66&$\begin{bmatrix} 1&0\\0&1\end{bmatrix}$ &
			$\begin{bmatrix} 0&0\\ 0&0\end{bmatrix}$& 3& Rem.~\ref{REMindex=3} &
			82&$\begin{bmatrix} 0&1\\0&1\end{bmatrix}$ &
			$\begin{bmatrix} 0&0\\0&0\end{bmatrix}$& $\infty$ &  Prop. \ref{ABall0all1} \\
			\hline
			67&&$\begin{bmatrix} 1&0\\ 0&0\end{bmatrix}$& 3 & Rem.~\ref{REMindex=3} &
			83&&$\begin{bmatrix} 1&0\\0&0\end{bmatrix}$& 4 & Rem.~\ref{REMindex=4}  \\
			\hline
			68&&$\begin{bmatrix} 0&1\\ 0&0\end{bmatrix}$& 3 & Rem.~\ref{REMindex=3} &
			84&&$\begin{bmatrix} 0&1\\0&0\end{bmatrix}$& $\infty$ & Rem~\ref{Remindex=inf} \\
			\hline
			69&&$\begin{bmatrix} 0&0\\ 1&0\end{bmatrix}$& 3 &Rem.~\ref{REMindex=3} &
			85&&$\begin{bmatrix} 0&0\\1&0\end{bmatrix}$& 4 & Rem.~\ref{REMindex=4} \\
			\hline		
			70&&$\begin{bmatrix} 0&0\\ 0&1\end{bmatrix}$& 3 & Rem.~\ref{REMindex=3} &
			86&&$\begin{bmatrix} 0&0\\0&1\end{bmatrix}$& $\infty$ & Rem~\ref{Remindex=inf} \\
			\hline		
			71&&$\begin{bmatrix} 1&1\\ 0&0\end{bmatrix}$& 3 &Rem.~\ref{REMindex=3} &
			87&&$\begin{bmatrix} 1&1\\0&0\end{bmatrix}$& 4 &Rem.~\ref{REMindex=4} \\
			\hline
			72&&$\begin{bmatrix} 1&0\\ 1&0\end{bmatrix}$& 3 &Rem.~\ref{REMindex=3} &
			88&&$\begin{bmatrix} 1&0\\1&0\end{bmatrix}$& $\infty$ & Prop. \ref{ABall0all1} \\
			\hline
			73&&$\begin{bmatrix} 1&0\\ 0&1\end{bmatrix}$& $\infty$ &  Cor.~\ref{A=B22} &
			89&&$\begin{bmatrix} 1&0\\0&1\end{bmatrix}$& 3 & Rem.~\ref{REMindex=3} \\
			\hline
			74&&$\begin{bmatrix} 0&1\\ 1&0\end{bmatrix}$& $\infty$ & Prop~\ref{case74}  &
			90&&$\begin{bmatrix} 0&1\\1&0\end{bmatrix}$& 3 &Rem.~\ref{REMindex=3} \\
			\hline
			75&&$\begin{bmatrix} 0&1\\ 0&1\end{bmatrix}$& 3 &Rem.~\ref{REMindex=3} &
			91&&$\begin{bmatrix} 0&1\\0&1 \end{bmatrix}$& $\infty$ &  Cor.~\ref{A=B22}\\
			\hline
			76&&$\begin{bmatrix} 0&0\\ 1&1\end{bmatrix}$& 3 & Rem.~\ref{REMindex=3} &
			92&&$\begin{bmatrix} 0&0\\1&1\end{bmatrix}$& 4 & Rem.~\ref{REMindex=4} \\
			\hline
			77&&$\begin{bmatrix} 1&1\\ 1&0\end{bmatrix}$& 3 & Rem.~\ref{REMindex=3} &
			93&&$\begin{bmatrix} 1&1\\1&0\end{bmatrix}$& 4 &Rem.~\ref{REMindex=4}  \\
			\hline
			78&&$\begin{bmatrix} 1&1\\ 0&1\end{bmatrix}$& 3 & Rem.~\ref{REMindex=3} &
			94&&$\begin{bmatrix} 1&1\\0&1\end{bmatrix}$& $\infty$& Rem~\ref{Remindex=inf} \\
			\hline
			79&&$\begin{bmatrix} 1&0\\ 1&1\end{bmatrix}$& 3 &Rem.~\ref{REMindex=3} &
			95&&$\begin{bmatrix} 1&0\\1&1\end{bmatrix}$& 4 & Rem.~\ref{REMindex=4} \\
			\hline
			80&&$\begin{bmatrix} 0&1\\ 1&1\end{bmatrix}$& 3 &Rem.~\ref{REMindex=3} &
			96&&$\begin{bmatrix} 0&1\\1&1\end{bmatrix}$& $\infty$ &   case 94 \\
			\hline
			81&&$\begin{bmatrix} 1&1\\ 1&1\end{bmatrix}$& 2 & Rem.~\ref{REMindex=2}&
			97&&$\begin{bmatrix} 1&1\\1&1\end{bmatrix}$& $\infty$ & Prop. \ref{ABall0all1} \\
			\hline\hline
			98&$\begin{bmatrix} 1&1\\0&1\end{bmatrix}$ &
			$\begin{bmatrix} 0&0\\ 0&0\end{bmatrix}$& 4 &Rem.~\ref{REMindex=4}  &
			114&$\begin{bmatrix} 1&1\\1&1\end{bmatrix}$&$\begin{bmatrix}0&0\\0&0\end{bmatrix}$&$\infty$& Prop. \ref{ABall0all1}  \\
			\hline
			99&&$\begin{bmatrix} 1&0\\ 0&0\end{bmatrix}$& 3&Rem.~\ref{REMindex=3} &
			115&&$\begin{bmatrix} 1&0\\ 0&0\end{bmatrix}$&5& Rem.~\ref{REMcase115-116-118}  \\
			\hline
			100&&$\begin{bmatrix} 0&1\\ 0&0\end{bmatrix}$&  $\infty$ & Rem~\ref{Remindex=inf}  &
			116&&$\begin{bmatrix} 0&1\\ 0&0\end{bmatrix}$&5& Rem.~\ref{REMcase115-116-118} \\
			\hline
			101&&$\begin{bmatrix} 0&0\\ 1&0\end{bmatrix}$& 2 & Rem.~\ref{REMindex=2} &
			117&&$\begin{bmatrix} 0&0\\ 1&0\end{bmatrix}$&5& Prop.~\ref{Propcase117} \\
			\hline		
			102&&$\begin{bmatrix} 0&0\\ 0&1\end{bmatrix}$& 3&Rem.~\ref{REMindex=3} &
			118&&$\begin{bmatrix} 0&0\\ 0&1\end{bmatrix}$&5& Rem.~\ref{REMcase115-116-118}  \\
			\hline		
			103&&$\begin{bmatrix} 1&1\\ 0&0\end{bmatrix}$& $\infty$ & Rem~\ref{Remindex=inf}  &
			119&&$\begin{bmatrix} 1&1\\ 0&0\end{bmatrix}$&$\infty$& Prop. \ref{ABall0all1} \\
			\hline
			104&&$\begin{bmatrix} 1&0\\ 1&0\end{bmatrix}$& 3 & Rem.~\ref{REMindex=3} &
			120&&$\begin{bmatrix} 1&0\\ 1&0\end{bmatrix}$&$\infty$& Prop. \ref{ABall0all1}   \\
			\hline
			105&&$\begin{bmatrix} 1&0\\ 0&1\end{bmatrix}$& 4 &Rem.~\ref{REMindex=4} &
			121&&$\begin{bmatrix} 1&0\\ 0&1\end{bmatrix}$& 3 & Prop.~\ref{Propcase121} \\
			\hline
			106&&$\begin{bmatrix} 0&1\\ 1&0\end{bmatrix}$& 3 &Rem.~\ref{REMindex=3} &
			122&&$\begin{bmatrix} 0&1\\ 1&0\end{bmatrix}$& 3 & Rem.~\ref{REMcase122} \\
			\hline
			107&&$\begin{bmatrix} 0&1\\ 0&1\end{bmatrix}$& $\infty$ & Rem~\ref{Remindex=inf} &
			123&&$\begin{bmatrix} 0&1\\ 0&1\end{bmatrix}$&$\infty$& Prop. \ref{ABall0all1}   \\
			\hline
			108&&$\begin{bmatrix} 0&0\\ 1&1\end{bmatrix}$& 3 & Rem.~\ref{REMindex=3} &
			124&&$\begin{bmatrix} 0&0\\ 1&1\end{bmatrix}$&$\infty$&Prop. \ref{ABall0all1}   \\
			\hline
			109&&$\begin{bmatrix} 1&1\\ 1&0\end{bmatrix}$& 3 & Rem.~\ref{REMindex=3}&
			125&&$\begin{bmatrix} 1&1\\ 1&0\end{bmatrix}$&$\infty$& Thm.~\ref{Thm100101->WR}  \\
			\hline
			110&&$\begin{bmatrix} 1&1\\ 0&1\end{bmatrix}$& $\infty$ &  Cor.~\ref{A=B22} &
			126&&$\begin{bmatrix} 1&1\\ 0&1\end{bmatrix}$&$\infty$& Thm.~\ref{Thm100101->WR}   \\
			\hline
			111&&$\begin{bmatrix} 1&0\\ 1&1\end{bmatrix}$& 3 & Rem.~\ref{REMindex=3}&
			127&&$\begin{bmatrix} 1&0\\ 1&1\end{bmatrix}$&$\infty$& Thm.~\ref{Thm100101->WR}   \\
			\hline
			112&&$\begin{bmatrix} 0&1\\ 1&1\end{bmatrix}$& 3 &Rem.~\ref{REMindex=3} &
			128&&$\begin{bmatrix} 0&1\\ 1&1\end{bmatrix}$&$\infty$& Thm.~\ref{Thm100101->WR}   \\
			\hline
			113&&$\begin{bmatrix} 1&1\\ 1&1\end{bmatrix}$& $\infty$ & Rem~\ref{Remindex=inf} &
			129&&$\begin{bmatrix} 1&1\\ 1&1\end{bmatrix}$&$\infty$& Prop. \ref{ABall0all1}  \\
			\hline
		\end{tabular}
	}
	\caption{ The remaining cases of the index of word-representability of infinite split graphs $G(A,B)$ for $2 \times 2$ matrices $A$ and $B$. }\label{tab:2x2-2}
	\end{table}

\end{center}

\section{Preliminaries}\label{prelim-sec}

\begin{definition}
		Let $M$ be a binary $m \times n$ matrix. Define $S(M)$ to be the matrix
		\[
		\left[ {\begin{array}{cc}
			L_n & M^T\\
			M & O_m\\
			\end{array} } \right]
		\]
where $O_m$ is the $m \times m$ zero matrix and $L_n$ is the $n \times n$ matrix such that all diagonal entries are $0$'s and all other entries are $1$'s.
\end{definition}

It is easy to see that for any binary $m \times n$ matrix $M$, $S(M)$ is the adjacency matrix of a split graph with the maximal clique of order $n$ or $n+1$. We denote the split graph by $G(M)$. Clearly, $M$ gives edges between the clique and independent set in $G(M)$,  and the order of the maximal clique depends on the existence of a $11\cdots1$ row in $M$.

\begin{example} \label{ExM} {\em 
	If $M=\left[ {\begin{array}{cccc}
		1&0&1&0\\
		0&1&1&0\\
		1&0&0&0\\
		0&0&0&1
		\end{array} } \right]$ then $$S(M)=\left[ {\begin{array}{cccccccc}
		0&1&1&1&1&0&1&0\\
		1&0&1&1&0&1&0&0\\
		1&1&0&1&1&1&0&0\\
		1&1&1&0&0&0&0&1\\
		1&0&1&0&0&0&0&0\\
		0&1&1&0&0&0&0&0\\
		1&0&0&0&0&0&0&0\\
		0&0&0&1&0&0&0&0
		\end{array} } \right]$$
is the adjacency matrix of the graph shown in Figure~\ref{ExampleSM}.
}
\end{example}

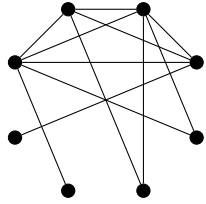
\begin{figure} 
\begin{center}

\begin{tikzpicture}[node distance=1cm,auto,main node/.style={fill,circle,draw,inner sep=0pt,minimum size=5pt}]

\node[main node] (1) {};
\node[main node] (2) [above right of=1] {};
\node[main node] (3) [right of=2] {};
\node[main node] (4) [below right of=3] {};
\node[main node] (5) [below of=4] {};
\node[main node] (6) [below left of=5] {};
\node[main node] (7) [left of=6] {};
\node[main node] (8) [above left of=7] {};

\path
(1) edge (2)
(1) edge (3)
(1) edge (4)
(2) edge (3)
(2) edge (4)
(3) edge (4)
(1) edge (5)
(1) edge (7)
(2) edge (6)
(3) edge (5)
(3) edge (6)
(4) edge (8);
\end{tikzpicture}

\caption{\label{ExampleSM}  The split graph $G(M)$ given by $S(M)$ in Example \ref{ExM}.}
\end{center}
\end{figure}

\begin{remark}\label{M=0}
If $M$ is a zero matrix, then $G(M)$ is  a disjoint union of a clique and isolated vertices and  is  word-representable, because the clique is semi-transitively (in fact, transitively) orientable and there are no other edges in $G(M)$. 
\end{remark}

The following lemma is Lemma 8 in \cite{KLMW17}.

\begin{lemma}[\cite{KLMW17}]\label{lem-neighb}
	Let $S_n=(E_{n-m},K_m)$ be a split graph with the maximum clique $K_m$, and a spit graph $S_{n+1}$ is obtained from $S_n$ by either adding a vertex of degree $0$ (to $E_{n-m}$), or adding a vertex of degree $1$ (to $E_{n-m}$), or by ``coping'' a vertex (either in $E_{n-m}$ or in $K_m$), that is, by adding a vertex whose neighbourhood is identical to the neighbourhood of a vertex in $S_n$. Then $S_n$ is word-representable if and only if $S_{n+1}$ is word-representable.
\end{lemma}

To analyze word-representability of a given split graph, by Lemma~\ref{lem-neighb} we can delete vertices of degree 0 and vertices of degree 1, as well as delete all but one vertex having the same neighborhood.

\begin{proposition}\label{Mall0all1}
Let $M$ be an $m\times n$ matrix. If every row, or every column of $M$ is of the form $00\cdots 0$ or $11\cdots 1$, then $G(M)$ is word-representable.
\end{proposition}

\begin{proof}
If every row of $M$ consists of all $0$'s or all $1$'s then in $G(M)$, each vertex in the independent set is an isolated vertex, or is connected to every vertex in the clique. By Lemma~\ref{lem-neighb}, word-representability of $G(M)$ is equivalent to word-representability of either the clique $K_{n+1}$, or the disjoint union of a clique ($K_n$ or $K_{n+1}$) and an isolated vertex, which are clearly semi-transitively orientable, and thus, by Theorem~\ref{key-thm}, word-representable.
		
	On the other hand, if every column of $M$ consists of all $0$'s or all $1$'s then the neighborhood of each vertex in the independent set is the same and, by Lemma~\ref{lem-neighb}, word-representability of  $G(M)$ is equivalent to word-representability of the  clique $K_{n}$ with a vertex $x$ connected to some, maybe none or all of clique's vertices. If w.l.o.g. $x$ is connected to vertices $1, 2, \ldots, p$, $0\leq p\leq n$, in $K_n$ formed by the vertices $1,2,\ldots,n$, then the word $x12\cdots px(p+1)(p+2)\cdots n$ represents the graph. 
\end{proof}

It is obvious that if $M^*$ is a matrix obtained by a row or column permutation of a matrix $M$, then $G(M^*)$ is a split graph obtained by relabelling the vertices of the graph $G(M)$. Hence we get the following lemma.

\begin{lemma} \label{Lem-permutrowcolumn}
	Let $M$ be an $m\times n$ binary matrix and $M^*$ is a matrix obtained from a sequence of row and/or column permutations of $M$.  Then, $G(M)$ is word-representable if and only if $G(M^*)$ is word-representable.
\end{lemma}

\begin{lemma} \label{all1row>>all0column}
	Let $M := [m_{ij}]_{m \times n}$ be an $m \times n$ binary matrix such that $m_{k1}=m_{k2}=\cdots=m_{kn}=1$ for some $k \in \{1,2,\ldots,m\}$. If 
	$$N= \begin{bmatrix}
	m_{11}&m_{12}&\cdots&m_{1n}&0\\	m_{21}&m_{22}&\cdots&m_{2n}&0\\
	\vdots & \vdots & & \vdots &\vdots \\
	m_{(k-1)1}&m_{(k-1)2}&\cdots&m_{(k-1)n}&0\\	m_{(k+1)1}&m_{(k+1)2}&\cdots&m_{(k+1)n}&0\\
	\vdots & \vdots & & \vdots & \vdots \\
	m_{m1}&m_{m2}&\cdots&m_{mn}&0
		
	\end{bmatrix},$$
	is an $(m-1) \times (n+1)$ binary matrix, then $G(M)$ is isomorphic to $G(N)$.
\end{lemma}
	
\begin{proof}
	Let $M^*$ be the matrix obtained from $M$ by making the $k^{th}$ row be the first row. That  is, 
	$$M^* = \begin{bmatrix}
	1&1& \cdots &1\\
	m_{11}&m_{12}&\cdots&m_{1n}\\	m_{21}&m_{22}&\cdots&m_{2n}\\
	\vdots & \vdots & & \vdots \\
	m_{(k-1)1}&m_{(k-1)2}&\cdots&m_{(k-1)n}\\	m_{(k+1)1}&m_{(k+1)2}&\cdots&m_{(k+1)n}\\
	\vdots & \vdots & & \vdots  \\
	m_{m1}&m_{m2}&\cdots&m_{mn}
	\end{bmatrix}.$$
	Since $M^*$ is obtained by reordering rows of $M$, $G(M^*)$ is obtained by relabeling the vertices of $G(M)$, and thus $G(M^*)$ is isomorphic to $G(M)$. Note that $S(M^*)=S(N)$, and so $G(M^*)$ and $G(N)$ are the same graph. Hence $G(M)$ is isomorphic to $G(N)$.
\end{proof}
 
It is known \cite{K17}  that there are no non-word-representable graphs of order less than 6 and the only non-word-representable graph on 6 vertices is the wheel graph $W_5$, which is not a split graph. Thus, $G(M)$ is word-representable if $M$ is an $m \times n$ matrix and $m+n  \leq 6$. In \cite{KLMW17}, it is shown that any split graph $S$ with maximum clique $K_4$ is word-representable if and only if $S$ does not contain the graphs $T_1$, $T_2$, $T_3$  and $T_4$ shown in Figure~\ref{T1T2T3T4} as induced subgraphs. As a corollary to this result, we have the following theorem.

\begin{figure}
	\begin{center}
		
		\begin{tabular}{ccccc}
			\begin{tikzpicture}[node distance=1cm,auto,main node/.style={fill,circle,draw,inner sep=0pt,minimum size=5pt}]
			
			\node[main node] (1) {};
			\node[main node] (2) [below left of=1] {};
			\node[main node] (3) [below right of=1] {};
			\node[main node] (4) [below left of=2] {};
			\node[main node] (5) [below right of=2] {};
			\node[main node] (6) [below right of=3] {};
			\node[main node] (7) [above of=5, yshift=-0.5cm] {};
			
			\node (8) [left of=2] {$T_1=$};
			
			\path
			(1) edge (4)
			(1) edge (6);
			
			\path
			(5) edge (2)
			(5) edge (3)
			(5) edge (4)
			(5) edge (6);
			
			\path
			(7) edge (2)
			(7) edge (3)
			(7) edge (5);
			
			\path
			(2) edge (3);
			\end{tikzpicture}

			& 
			
			\ \ \ 
			
			&
			
			\begin{tikzpicture}[node distance=1cm,auto,main node/.style={fill,circle,draw,inner sep=0pt,minimum size=5pt}]
			
			\node[main node] (1) {};
			\node[main node] (2) [below of=1] {};
			\node[main node] (3) [below left of=2, xshift=-0.5cm] {};
			\node[main node] (4) [below right of=2, xshift=0.5cm] {};
			\node[main node] (5) [left of=2, xshift=0.6cm] {};
			\node[main node] (6) [below of=2, yshift=0.7cm] {};
			\node[main node] (7) [above right of=2, xshift=-0.5cm, yshift=-0.5cm] {};
			\node (8) [left of=2,  xshift=-0.5cm] {$T_2=$};
			
			\path
			(1) edge (2)
			(1) edge (4)
			(1) edge (7)
			(1) edge (3);
			
			\path
			(2) edge (3)
			(2) edge (4)
			(2) edge (5)
			(2) edge (6)
			(2) edge (7);
			
			\path
			(3) edge (4)
			(3) edge (5);
			
			\path
			(6) edge (4);
			
			\end{tikzpicture}
			\\
			\begin{tikzpicture}[node distance=1cm,auto,main node/.style={fill,circle,draw,inner sep=0pt,minimum size=5pt}]
			
			\node[main node] (1) {};
			\node[main node] (2) [below of=1] {};
			\node[main node] (3) [below left of=2,xshift=-0.5cm] {};
			\node[main node] (4) [below right of=2,xshift=0.5cm] {};
			\node[main node] (5) [below of=2,yshift=0.6cm] {};
			\node[main node] (6) [left of=2,yshift=0.1cm,xshift=0.6cm] {};
			\node[main node] (7) [right of=2,yshift=0.1cm,xshift=-0.6cm] {};
			
			\node [above left of=3] {$T_3=$};
			
			\path
			(2) edge (6)
			(2) edge (7)
			(5) edge (3)
			(5) edge (4)
			(5) edge (2)
			(1) edge (3)
			(1) edge (4)
			(6) edge (3)
			(7) edge (4)
			(1) edge (6)
			(1) edge (7)
			(2) edge (3)
			(2) edge (4)
			(4) edge (3)
			(1) edge (2);
			\end{tikzpicture}
			
			& 
			
			\ \ \ 
			
			& 	
			\begin{tikzpicture}[node distance=1cm,auto,main node/.style={fill,circle,draw,inner sep=0pt,minimum size=5pt}]
			
			\node[main node] (1) {};
			\node[main node] (2) [right of=1] {};
			\node[main node] (3) [below of=2] {};
			\node[main node] (4) [left of=3] {};
			\node[main node] (5) [above right of=1,xshift=-0.2cm] {};
			\node[main node] (6) [above right of=2] {};
			\node[main node] (7) [below left of=1,yshift=0.2cm] {};
			\node[main node] (8) [below left of=4] {};
			
			\node [above left of=7,yshift=-0.5cm] {$T_4=$};
			
			\path
			(1) edge (2)
			(1) edge (3)
			(1) edge (4)
			(1) edge (5)
			(1) edge (6)
			(1) edge (7)
			(1) edge (8);
			
			\path
			(2) edge (3)
			(2) edge (4)
			(2) edge (5)
			(2) edge (6);
			
			\path
			(3) edge (4)
			(3) edge (6)
			(3) edge (8);
			
			\path
			(4) edge (8)
			(4) edge (7);
			
			\end{tikzpicture}
		\end{tabular}
		
		\caption{\label{T1T2T3T4} Non-word-representable split graphs $T_1,T_2,T_3,T_4$.} 
	\end{center}
\end{figure}
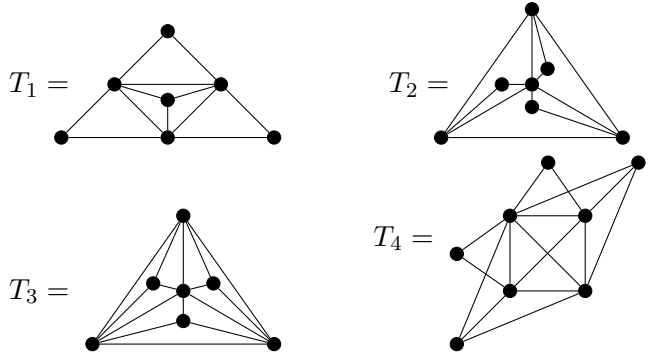

\begin{theorem} \label{N-W-R-44}
	Let $A$ be an $m \times 4$ binary matrix without all $1$'s rows. Then, $G(A)$ is word-representable if and only if the rows and columns of $A$ cannot be permuted to be a matrix containing 
	${\tiny \begin{bmatrix}
		1&1&0&0 \\	1&0&1&0 \\
		0&1&1&0 
%		\\	x_1&x_2&x_3&x_4
		\end{bmatrix} }$, 
	${\tiny \begin{bmatrix}
		1&1&0&0 \\	1&0&1&0 \\
		1&0&0&1 
%		\\	x_1&x_2&x_3&x_4
		\end{bmatrix} }$, 
	${\tiny \begin{bmatrix}
		1&0&1&1 \\	1&1&0&1 \\
		1&1&1&0 
%		\\	x_1&x_2&x_3&x_4
		\end{bmatrix} }$ or 
	${\tiny \begin{bmatrix}
		1&1&1&0 \\	1&1&0&0 \\
		1&0&0&1 \\	1&0&1&1 
		\end{bmatrix} }$  as a submatrix.
\end{theorem}

Let $x^n$ denotes $xx\cdots x$ where $x$ is repeated $n$ times. The next theorem gives a sufficient condition for word-representability of a given graph $G(M)$.

\begin{theorem}\label{Thm-permuterow010}
	Let $M$ be an $m \times n$ binary matrix. If there is a sequence of column permutations of $M$ giving the matrix such that each row of $M$ is of the form  $0^r  1^s  0^t$ for some non-negative integers $r,s,t$, then $G(M)$ is word-representable.
\end{theorem}
\begin{proof}
	Assume that $M^*$ is the matrix obtained from a sequence of column permutations of $M$ and each row of $M^*$ is of the form $0^r  1^s  0^t$ where $r,s,t$ are non-negative integers. Let the $i^{th}$ row/column of the adjacency matrix $S(M^*)$ correspond to vertex $i$ in $G(M^*)$. So the clique $C$ in $G(M^*)$ contains vertices $1,2,\dots,n$ and the independent set $I$ in $G(M^*)$ contains vertices $n+1,n+2, \dots , n+m$. Assign the orientation of edges in $G(M^*)$ as $i \rightarrow j$ if and only if $i<j$. We have that $1 \rightarrow 2 \rightarrow \cdots \rightarrow n$ is the longest path in the transitively oriented $C$, and the edges between $C$ and $I$ are oriented from $C$ to $I$. Thus, each edge in the independent set is of type B, and we are done by Theorems~\ref{key-thm} and \ref{main-orientation}.	
\end{proof}

We have the following important generalization of Theorem~\ref{Thm-permuterow010}.

\begin{theorem} \label{Thm-genpermutecolumn}
	Let $M$ be an $m \times n$ binary matrix without all $1$'s rows. The split graph $G(M)$ is word-representable if and only if $M$ satisfies the following conditions:
	\begin{itemize}
		\item[{\em (i)}] there is a sequence of column permutations of $M$ giving a matrix $M^*$ where every row is of the form $ 0^r  1^s  0^t $ or $ 1^r  0^s 1^t $ for some nonnegative integers $r$, $s$, $t$, and
		\item[{\em (ii)}] for any row of $M^*$ of the form $ 1^a 0^b  1^c $ for some positive integer $a,b,c$, there is no other row not of the form $ 1^a 0^b  1^c $ in which the entries in positions $a$ and $a+b+1$ are $1$'s.
	\end{itemize}
\end{theorem}

\begin{proof} ``$\Leftarrow$.'' Assign the orientation of edges in $G(M^*)$ as $i \rightarrow j$ if $i<j$ {\em except} if $j>n$ (i.e.\ $j$ is a vertex in the independent set) and the row in $M^*$ corresponding to $j$ is of the form $1^r  0^s 1^t $, in which case we still orient $i\rightarrow j$ for $1\leq i\leq r$ but $j\rightarrow i$ for $r+s+1\leq i\leq n$. The vertices in the independent set will then be of types B and C, and taking into account condition (ii),  Theorems~\ref{key-thm} and \ref{main-orientation} can be applied to see that $G(M^*)$ is word-representable, and thus $G(M)$ is word-representable by Lemma~\ref{Lem-permutrowcolumn}. \\

\noindent
``$\Rightarrow$.'' By Theorem~\ref{key-thm}, $G(M)$ admits a semi-transitive orientation. By Theorem~\ref{main-orientation}, under this orientation the clique is oriented transitively, and we can rename the vertices of the clique, if necessary so that the longest path would be formed by $1\rightarrow 2\rightarrow\cdots\rightarrow n$. Note that renaming vertices in the clique corresponds to permuting columns in $M$ giving $M^*$. But then, conditions (ii) and (iii) in Theorem~\ref{main-orientation} give conditions (i) and (ii) in this theorem.
\end{proof}

\begin{remark}
	If $M$ has an all $1$'s row, we can see that Theorems~\ref{N-W-R-44} and~\ref{Thm-genpermutecolumn} cannot be applied. However, we can use Theorem~\ref{all1row>>all0column} to change $M$ into an $(m-1) \times (n+1)$ matrix $N$ which does not contain all $1$'s row. So we can apply the theorems to matrix $N$ instead of $M$ because $G(N)$ is isomorphic to $G(M)$. This observation also applies to  Corollary~\ref{N-W-R-AA44} below.  
\end{remark}

We can see that Theorem \ref{Thm-genpermutecolumn} allows us to answer the question on word-representability of $G(M)$ by looking at permutations of columns in $M$. Let $M= \left[m_{ij}\right]_{m \times n}$ be an $m \times n$ matrix and $\rho=\rho_1\rho_2 \cdots \rho_n$ is a permutation of $\{ 1,2, \ldots,n \}$ written in one-line notation. We say that $$M^*=\begin{bmatrix}
m_{1\rho_1}&m_{1\rho_2}&\cdots&m_{1\rho_n}\\
m_{2\rho_1}&m_{2\rho_2}&\cdots&m_{2\rho_n}\\
\vdots&\vdots&\ddots&\vdots\\
m_{m\rho_1}&m_{m\rho_2}&\cdots&m_{m\rho_n}\\
\end{bmatrix}$$
is {\it the matrix obtained from reordering columns of $M$ in the order given by $\rho$}. The key approach given by Theorem \ref{Thm-genpermutecolumn} is finding a permutation $\rho$ that turns each row of $M^*$ into the form $0^r1^s0^t$ or $1^r0^s1^t$ (so, all 1's in $M^*$ are cyclically consecutive). Interestingly, to prove word-representability results in this paper, only rows of the form $0^r1^s0^t$ are used, so that condition (ii) in Theorem~\ref{Thm-genpermutecolumn} is not applicable.

\begin{example}\label{Ex-8x8} {\em  
In the matrix $M=\begin{bmatrix}
1&0&0&1&0&1&1&0\\
0&0&0&0&0&0&0&0\\
1&0&1&0&1&0&1&0\\
0&0&0&0&0&0&0&0\\
0&1&1&0&0&1&1&0\\
0&0&0&0&0&0&0&0\\
1&0&0&1&0&1&0&1\\
0&0&0&0&0&0&0&0
\end{bmatrix}$ \\[2mm]
	we can ignore rows $2, 4, 6$ and $8$ because all entries in these rows are zero. Then we need to find a permutation $\rho = \rho_1\rho_2 \cdots \rho_8$ making
	\begin{itemize}
		\item 	columns $1, 4, 6$ and $7$ be (cyclically) consecutive in rows $1$;
		\item columns $1, 3, 5$ and $7$ be (cyclically) consecutive in row $3$ and $7$;
		\item  columns $2, 3, 6$ and $7$ be (cyclically) consecutive in row $5$.
	\end{itemize}
	It can be implied from the first and the second bullet points that 1 and 7 must be consecutive in $\rho$ and then, w.l.o.g., 4 and 6 are next to the left of these numbers and 3 and 5 are next to the right of them (cyclically). Hence $\rho$ contains $$\{ 4,6 \}, \{ 1,7 \}, \{ 3,5 \}$$ where numbers in $\{ \}$ are consecutive in $\rho$ but are in some unknown to us order.  But then, we get a contradiction with the second bullet point. Hence, there is no such $\rho$ and $G(M)$ is non-word-representable by Theorem~\ref{Thm-genpermutecolumn}.
}
\end{example}

\section{General results on split graphs generated by morphisms}\label{general-sec}

In this section, we discuss rather general results on split graphs generated by morphisms, thus preparing ourselves for a classification of the case of $2\times 2$ matrices coming in the next section.

\begin{definition}
	Let  $A,B$ be $m\times n$ binary matrices. The matrix $M^k(A,B)$ is said to be the $k^{th}$-iteration of the $2$-dimensional morphism applied to the $1 \times 1$ matrix $\left[ 0 \right]$ which maps $\left[ 0 \right] \rightarrow A$ and $\left[ 1 \right] \rightarrow B$. Moreover, we write $S^k(A,B)$ for the matrix $S(M^k(A,B))$ and $G^k(A,B)$ for the graph with the adjacency matrix $S^k(A,B)$.
\end{definition}

\begin{example} \label{ExAB}
	Let $A=\begin{bmatrix}1&0\\0&1\end{bmatrix}$ and $B=\begin{bmatrix}0&1\\1&0	\end{bmatrix}$. Then we have
	\begin{center}
		\begin{tabular}{lll}
			$M^0(A,B)=\begin{bmatrix}0\end{bmatrix}$, & 
			$M^1(A,B)=\begin{bmatrix}1&0\\0&1\end{bmatrix}$, &
			$M^2(A,B)=\begin{bmatrix}0&1&1&0\\1&0&0&1\\1&0&0&1\\0&1&1&0\end{bmatrix}$.
		\end{tabular}
	\end{center}
Then, $S^2(A,B) = \begin{bmatrix} 0&1&1&1&0&1&1&0\\1&0&1&1&1&0&0&1\\1&1&0&1&1&0&0&1\\1&1&1&0&0&1&1&0\\0&1&1&0&0&0&0&0\\1&0&0&1&0&0&0&0\\1&0&0&1&0&0&0&0\\0&1&1&0&0&0&0&0\end{bmatrix} $ and $G^2(A,B)$ is shown in Figure~\ref{ExampleSAB}.
\end{example}

\begin{figure} 
	\begin{center}
		
		\begin{tikzpicture}[node distance=1cm,auto,main node/.style={fill,circle,draw,inner sep=0pt,minimum size=5pt}]
		
		\node[main node] (1) {};
		\node[main node] (2) [above right of=1] {};
		\node[main node] (3) [right of=2] {};
		\node[main node] (4) [below right of=3] {};
		\node[main node] (5) [below of=4] {};
		\node[main node] (6) [below left of=5] {};
		\node[main node] (7) [left of=6] {};
		\node[main node] (8) [above left of=7] {};
		
		\path
		(1) edge (2)
		(1) edge (3)
		(1) edge (4)
		(2) edge (3)
		(2) edge (4)
		(3) edge (4)
		(2) edge (5)
		(3) edge (5)
		(2) edge (8)
		(3) edge (8)
		(1) edge (7)
		(4) edge (7)
		(1) edge (6)
		(4) edge (6);
		\end{tikzpicture}
		
		\caption{\label{ExampleSAB}  The split graph $G^2(A,B)$ corresponding to the adjacency matrix $S^2(A,B)$ in Example \ref{ExAB}.}
	\end{center}
\end{figure}
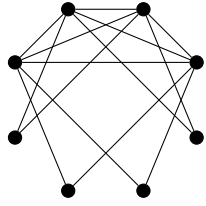

\begin{remark}\label{A=0}
If $A$ is a zero matrix, then $M^k(A,B)$ is always a zero matrix for any $m \times n$ matrix $B$ and positive integer $k$. So, by Remark~\ref{M=0}, $G^k(A,B)$ is word-representable. 
\end{remark}

\subsection{The case of $A=B$}

In this case, both $\left[0\right]$ and $\left[1\right]$ are mapped to the same matrix~$A$, so if $A$ is an $m\times n$ binary matrix, then
$$S^k(A,A)=\left[ {\begin{array}{cc}
	L_{n^k} & A_k^T\\
	A_k & O_{m^k}
	\end{array} } \right]
\text{~~~~where~~} A_k =  \underbrace{ \left[ \begin{array}{cccc}
		A&A& \cdots & A \\
		A&A& \cdots & A \\
		\vdots& \vdots& \ddots& \vdots\\
		A&A& \cdots & A 
		\end{array}  \right]}_{n^k \text{~columns}}. $$ 
Clearly, $A_k$ is an $n^k \times m^k$ matrix and $S^k(A,A) = S(A_k)$, so $G^k(A,A)$ is isomorphic to $G(A_k)$.

\begin{theorem}\label{GKAA}
	Let $A$ be an $m
		\times n$ binary matrix. For $k \geq 1$, $G^k(A,A)$ is word-representable if and only if $G(A)$ is word-representable.
\end{theorem}
\begin{proof}
	Firstly, we label a vertex of $G^k(A,A)$ by $i$ if it is represented by the $i^{th}$ column/row in $S^k(A,A)$.  Note that rows $i,m+i,2m+i, \dots , (m^{k-1}-1)m+i$ in $A_k$ are identical for any $i \in \{ 1,2,\dots ,m \}$, and columns $j,n+j,2n+j,\dots,(n^{k-1}-1)n+j$ in $A_k$ are also identical for any $j \in \{ 1,2,\dots,n \}$.  So, for any $i \in \{1,2,\dots, m\}$, the vertices of $G^k(A,A)$ in $R_i:= \{i+n^k,m+i+n^k,2m+i+n^k, \dots , (m^{k-1}-1)m+i+n^k\}$ have the same neighborhoods. Similarly, any two vertices of $G^k(A,A)$ in $C_j:=\{j,n+j,2n+j,\dots,(n^{k-1}-1)n+j\}$ are connected to the same vertices in the independent set for any $j \in \{1,2,\dots,n\}$. Thus, by Lemma~\ref{lem-neighb}, $G^k(A,A)$ is word-representable if and only if the graph $G$ obtained by deleting all vertices but the smallest one in $R_i$ and $C_j$ for all $i,j\in\{1,2,\ldots,n\}$ is word-representable. But $G$ is exactly $G(A)$, which complete the proof. 
\end{proof}

\begin{corollary} \label{A=B22}
	If $A$ is an $m \times n$ binary matrix such that $m+n \leq 6$, then $G^k(A,A)$ is word-representable for any $k\geq 0$.
\end{corollary}
\begin{proof}
	Since the smallest non-word-representable split graph is of order~$7$, all split graphs of orders less than $7$ are word-representable. Hence $G^k(A,A)$ is word-representable for any $m \times n$ matrix $A$ where $m+n \leq 6$.
\end{proof}

Moreover, together with Theorem~\ref{N-W-R-44}, we have the following result.

\begin{corollary} \label{N-W-R-AA44}
	Let $A$ be an $m \times 4$ binary matrix with no all $1$'s row. For any integer $k$, the graph $G^k(A,A)$ is word-representable if and only if the rows and columns of $A$ cannot be permuted to be the matrix containing 
			${\tiny \begin{bmatrix}
			1&1&0&0 \\	1&0&1&0 \\
			0&1&1&0 
%			\\	x_1&x_2&x_3&x_4
			\end{bmatrix} }$, 
		${\tiny \begin{bmatrix}
			1&1&0&0 \\	1&0&1&0 \\
			1&0&0&1 
%			\\	x_1&x_2&x_3&x_4
			\end{bmatrix} }$, 
		${\tiny \begin{bmatrix}
			1&0&1&1 \\	1&1&0&1 \\
			1&1&1&0 
%			\\	x_1&x_2&x_3&x_4
			\end{bmatrix} }$ or 
		${\tiny \begin{bmatrix}
			1&1&1&0 \\	1&1&0&0 \\
			1&0&0&1 \\	1&0&1&1 
			\end{bmatrix} }$ as a submatrix.
\end{corollary}

\subsection{The case of $A \neq B$}

In what follows, $A$ and $B$ can be distinct.  

\begin{proposition}\label{ABall0all1}
If every row, or every column, in $m\times n$ matrices $A$ and $B$ is either $0^n$ or  $1^n$, then $G^k(A,B)$ is word-representable for any $k\geq 0$. 
\end{proposition}

\begin{proof} If every row (resp., column) in $A$ and $B$ is either $0^n$ or  $1^n$, then every row (resp., column) in $M^k(A,B)$ is either $0^{n^k}$ or  $1^{n^k}$, so by Proposition~\ref{Mall0all1}, $G^k(A,B)$ is word-representable. \end{proof}

\begin{theorem} \label{Thm-ABcolpermute}
	Let $A$ and $B$ be $m \times n$ binary matrices. Suppose that $A^*$ and $B^*$ are the matrices obtained from reordering columns of $A$ and $B$, respectively, in order given by a permutation $\sigma=\sigma_1\sigma_2\cdots\sigma_n$. Then $G^k(A,B)$ is word-representable if and only if $G^k(A^*,B^*)$ is word-representable for any $k\geq 0$.
\end{theorem}
\begin{proof}
The case $k=0$ is trivial, so assume that $k\geq 1$. We claim that $M^k(A^*,B^*)$ is obtained from a permutation of columns in $M^k(A,B)$. We will prove the claim by induction on $k$. Note that $M^1(A,B)=A$ and $M^1(A^*,B^*)=A^*$. So $M^1(A^*,B^*)$ is the matrix obtained from reordering columns of $M^1(A,B)$. Suppose that $l$ is a positive integer and $M^l(A^*,B^*)$ is the matrix obtained from reordering columns of $M^l(A,B)$ in order given by a permutation $\tau=\tau_1\tau_2\cdots \tau_{n^l}$. Let $M^l(A,B)= \begin{bmatrix} C_1 & C_2 & \cdots & C_{n^l}	\end{bmatrix}$ where $C_i$ is the $i^{th}$ column of $M^l(A,B)$. Then $M^l(A^*,B^*)= \begin{bmatrix} C_{\tau_1} & C_{\tau_2} & \cdots & C_{\tau_{n^l}} \end{bmatrix}$. For the next iteration of morphism, each column $C_i$ of $M^l(A,B)$ is mapped to $n$ columns $C_{i,1}, C_{i,2}, \ldots, C_{i,n}$, and each column $C_{\tau_i}$ of $M^l(A^*,B^*)$ is mapped to $n$ columns  
$C_{\tau_i,\sigma_1}, C_{\tau_i,\sigma_2}, \ldots, C_{\tau_i,\sigma_n}$. So we have $$M^{l+1}(A,B)= \left[ C_{1,1} ~ C_{1,2} ~ \cdots  C_{1,n} ~ \cdots ~ C_{n^l,1} ~ C_{n^l,2} ~ \cdots ~ C_{n^l,n} \right]$$ and 
	\begin{align*}
	M^{l+1}(A^*,B^*) = [&~ C_{\tau_1,\sigma_1} ~ C_{\tau_1,\sigma_2} ~ \cdots ~ C_{\tau_1,\sigma_n} ~    \cdots ~ C_{\tau_{n^l},\sigma_1} \\ &C_{\tau_{n^l},\sigma_2}  ~ \cdots ~ C_{\tau_{n^l},\sigma_n}~].
	\end{align*}
A group of columns $C_{i,1},C_{i,2},\ldots,C_{i,n}$ is called	{\em block} $B_i$. Firstly, we can see that reordering the blocks $B_1,B_2,\ldots,B_{n^l}$ of $M^{l+1}(A,B)$ in order given by $\tau$, and then reordering columns in every block $B_i$ in order given by $\sigma$, yields the matrix $M^{l+1}(A^*,B^*)$. Thus, $M^{l+1}(A^*,B^*)$ is obtained by a column permutation of $M^{l+1}(A,B)$ and our claim is true. Hence, by Lemma~\ref{Lem-permutrowcolumn}, $G^k(A,B)$ is word-representable if and only if $G^k(A^*,B^*)$ is word-representable for any positive integer $k$. 
\end{proof}

Next theorem is a natural extension of Theorem~\ref{Thm-ABcolpermute} to the case of row permutations, and it can be proved in a similar way to the proof of Theorem~\ref{Thm-ABcolpermute}, so we omit the proof.

\begin{theorem}\label{Thm-ABcol&rowpermute}
	Let $A$ and $B$ be $m \times n$ binary matrices. Suppose that $A^*$ and $B^*$ are the matrices obtained from reordering rows of $A$ and $B$, respectively, in order given by the same permutation. Then $G^k(A,B)$ is word-representable if and only if $G^k(A^*,B^*)$ is word-representable for any $k\geq 0$.
\end{theorem}

So, we can reorder rows and columns of given matrices $A$ and $B$ while preserving the word-representability of $G^k(A,B)$. If $A$ contains at least one $0$, we can reorder rows and columns of $A$ to make  the leftmost bottom entry be a $0$ (the matrix $B$ will be changed by the same permutation of rows and columns as those applied to $A$). Thus, in what follows, if $A$ is not all-one matrix, w.l.o.g. we can assume that the leftmost bottom entry of $A$ is always 0. Then, the $m\times n$ leftmost bottom submatrix of $M^2(A,B)$ is $A$  since $M^1(A,B)=A$. Moreover, the $m^{k-1} \times n^{k-1}$ leftmost bottom submatrix of $M^{k}(A,B)$ is $M^{k-1}(A,B)$. Thus, the limit $\lim_{k\to\infty}M^k(A,B)$, called a {\em fixed point of the morphism}, is well-defined. So, we have that $G^i(A,B)$ is an induced subgraph of $G^k(A,B)$ if $i \leq k$ and the notion of the infinite split graph $G(A,B)$ is well-defined in the case when $A$ has a 0 as the leftmost bottom entry. So we are interested in the smallest integer $l$ (possibly non-existing) that $G^l(A,B)$ is non-word-representable for given $A$ and $B$ (then $G^i(A,B)$ is non-word-representable for $i\geq l$).

\begin{definition}\label{IWR-def} Suppose that a matrix $A$ has a $0$ as the leftmost bottom entry. The {\em index of word-representability} $\IWR$$(A,B)$ of an infinite split  graph $G(A,B)$ is the smallest integer $l$ such that $G^l(A,B)$ is non-word-representable. If such $l$ does not exist, that is, if $G^l(A,B)$ is word-representable for all $l$, then $l:=\infty$.  If the leftmost bottom entry of $A$ is $1$ (so that the $\lim_{k\to\infty}M^k(A,B)$ may not be well-defined as the sequence of graphs $G^k(A,B)$, for $k\geq 0$, may not be a chain of induced subgraphs) then  $\IWR$$(A,B)$ is still defined in the same way even though $G(A,B)$ may not be defined. \end{definition}

Note that since $G^0(A,B)$ is a graph with one vertex, we have $\IWR$$(A,B) \geq 1$.  Even though Definition~\ref{IWR-def} is very similar to the respective definition of the index of word-representability of an infinite {\em Toeplitz graph} in \cite{WRofToeplitz} (where the index in our context would be the maximum $l$ such that $G^l(A,B)$ is word-representable), it is more flexible as it makes sense in the situation when the leftmost bottom entry of $A$ is 1.  

\begin{theorem}\label{thm-submat}
	Let $A= \left[ a_{ij} \right]$ and $B= \left[ b_{ij} \right]$ be $m \times n$ binary matrices, and 

\noindent
$C = \begin{bmatrix}
	a_{p_1q_1} & a_{p_1q_2} & \cdots & a_{p_1q_t} \\ a_{p_2q_1} & a_{p_2q_2} & \cdots & a_{p_2q_t} \\ \vdots&\vdots&\ddots&\vdots \\ a_{p_sq_1} & a_{p_sq_2} & \cdots & a_{p_sq_t} \end{bmatrix} $ and $D = \begin{bmatrix}
	b_{p_1q_1} & b_{p_1q_2} & \cdots & b_{p_1q_t} \\ b_{p_2q_1} & b_{p_2q_2} & \cdots & b_{p_2q_t} \\ \vdots&\vdots&\ddots&\vdots \\ b_{p_sq_1} & b_{p_sq_2} & \cdots & b_{p_sq_t} \end{bmatrix} $ \\

\noindent
be $s \times t$ submatrices of $A$ and $B$, respectively, where  $1 \leq p_1 < p_2< \cdots < p_s \leq m$ and $1 \leq q_1 < q_2< \cdots < q_t  \leq n$. For any positive integer $k$, if $G^k(C,D)$ is non-word-representable, then $G^k(A,B)$ is non-word-representable. 
\end{theorem} 

\begin{proof}
	First, we will prove by induction that $M^k(C,D)$ is a submatrix of $M^k(A,B)$ for any positive integer $k$.  It is obvious that $M^1(C,D)= C$ is a submatrix of $M^1(A,B) = A$. Let $l$ be a positive integer such that $M^l(C,D)$ is a submatrix of $M^l(A,B)$ on the columns $c_1,c_2,\ldots,c_{t^l}$ and rows $r_1,r_2,\ldots,r_{s^l}$. For the next iteration of morphism, $M^{l+1}(A,B)$ is formed by replacing each entry of $M^l(A,B)$ with either $A$ or $B$. So the columns $(c_i-1)n+p_j$ for $1 \leq i \leq t^l, 1 \leq j \leq s$, and rows $(r_i-1)m+q_j$ for $1 \leq i \leq s^l$, $1 \leq j \leq t$, form the matrix $M^{l+1}(C,D)$. Hence $M^k(C,D)$ is a submatrix of $M^k(A,B)$ for any $k \geq 1$. Therefore, $G^k(A,B)$ contains $G^k(C,D)$ as an induced subgraph for $k \geq 1$. As the property of word-representability is hereditary, we have that non-word-representability of $G^k(C,D)$ implies non-word-representability of $G^k(A,B)$.
\end{proof}

Theorem~\ref{thm-submat} gives a useful tool to study non-word-representability of $G^k(A,B)$ for larger $A$ and $B$. Indeed, a starting point to justify suspected non-word-representability of $G^k(A,B)$ can be analysis of smaller submatrices of $A$ and $B$. This is one of our motivation points to conduct a systematic study of $\IWR$$(A,B)$ for $2 \times 2$ matrices, to be done in the next section, as they are smallest submatrices that can be used to show non-word-representability of $G^k(A,B)$ for some $A$, $B$ and~$k$.

\section{Classification of word-representable split graphs defined by iteration of morphisms using two $2 \times 2$ matrices}\label{class-2x2-sec}

A summary of our classification of  word-representability of $G^k(A,B)$ for $2 \times 2$ matrices $A$ and $B$ can be found in Tables~\ref{tab:2x2-1} and~\ref{tab:2x2-2}, where the index of word-representability $\IWR$$(A,B)$ is given along with a reference, or a comment to the respective result.

\subsection{The case of $A$ is not all-one matrix}

For any $2 \times 2$ matrices $A$ and $B$, the graph $G^1(A,B)$ is a split graph of order 4 which is always word-representable. Then, $\IWR$$(A,B) \geq 2$. However,  $2 \times 2$ matrices $A$ and $B$ such that $G^2(A,B)$ is non-word-representable can be found.

\begin{proposition}\label{caseindex=2}
		For $A=\begin{bmatrix} 1&0\\0&0 \end{bmatrix}$ and $B=\begin{bmatrix} 0&1\\1&1 \end{bmatrix}$, $\IWR$$(A,B) = 2$.
\end{proposition}
\begin{proof}
	We have $M^2(A,B)= \begin{bmatrix}
	0&1&1&0\\1&1&0&0\\1&0&1&0\\0&0&0&0
	\end{bmatrix}$. Reordering columns of $M^2(A,B)$ in order given by the permutation $2314$ yields the matrix $\begin{bmatrix}
	1&1&0&0\\1&0&1&0\\0&1&1&0\\0&0&0&0
	\end{bmatrix}$. So, by Theorem~\ref{N-W-R-44}, $G^2(A,B)$ is non-word-representable ($G^2(A,B)$ contains $T_1$ as an induced subgraph). 
\end{proof}

\begin{remark}\label{REMindex=2}
Permuting rows or/and columns in $A$ and $B$ similarly to Proposition~\ref{caseindex=2}, we see that $\IWR$$(A,B)=2$ for $A$ and $B$ in Cases 31, 45, 81 and 101 in Tables~\ref{tab:2x2-1} and~\ref{tab:2x2-2} (in Case 81 $G^2(A,B)$ contains $T_3$, and in the other cases $G^2(A,B)$ contains $T_1$). 
\end{remark}

\begin{proposition}\label{caseindex=3}
	For $A=\begin{bmatrix} 1&0\\0&0 \end{bmatrix}$ and $B=\begin{bmatrix} 0&1\\0&0 \end{bmatrix}$, $\IWR$$(A,B) = 3$.
\end{proposition}
\begin{proof}
	We have $$M^2(A,B)= \begin{bmatrix}
	0&1&1&0\\0&0&0&0\\1&0&1&0\\0&0&0&0
	\end{bmatrix} \text{~and~} M^3(A,B)=\begin{bmatrix}
	1&0&0&1&0&1&1&0\\
	0&0&0&0&0&0&0&0\\
	1&0&1&0&1&0&1&0\\
	0&0&0&0&0&0&0&0\\
	0&1&1&0&0&1&1&0\\
	0&0&0&0&0&0&0&0\\
	1&0&0&1&0&1&0&1\\
	0&0&0&0&0&0&0&0
	\end{bmatrix}.$$ Reordering columns of $M^2(A,B)$ in order given by  the permutation $4231$ yields the matrix $\begin{bmatrix}
	0&1&1&0\\0&0&0&0\\0&0&1&1\\0&0&0&0
	\end{bmatrix}$. By Theorem \ref{Thm-permuterow010}, we have $G^2(A,B)$ is word-representable. However, we have shown in Example \ref{Ex-8x8} that $G^3(A,B)$ is non-word-representable. So $\IWR$$(A,B) = 3$. 
\end{proof}

\begin{remark} \label{REMindex=3}
	Proposition~\ref{caseindex=3} gives Case 4 in Tables~\ref{tab:2x2-1}. In each of Cases 5, 10, 11, 12, 14, 15, 17, 57, 66, 67, 68, 71, 72, 77, 78, 89, 99, 102, 104, 106, 108, 109, 111 and 112 in Tables~\ref{tab:2x2-1} and~\ref{tab:2x2-2}, $M^2(A,B)$ has a permutation satisfying the conditions of Theorem~\ref{Thm-ABcolpermute} and $M^3(A,B)$ does not (similarly to Proposition~\ref{caseindex=3}). So $\IWR$$(A,B)=3$ for $A$ and $B$ in these cases. Moreover, by Theorem~\ref{Thm-ABcolpermute}, column and row permutations of $A$ and $B$ give the same $\IWR$. Consequently, we also have $\IWR$$(A,B)=3$ for $A$ and $B$ in Cases 19, 22, 24, 25, 28, 29, 32, 33, 36, 37, 39, 40, 42, 46, 47, 49, 58, 69, 70, 75, 76, 79, 80 and 90.
\end{remark}

\begin{proposition}\label{caseindex=4}
	For $A=\begin{bmatrix} 1&0\\0&0 \end{bmatrix}$ and $B=\begin{bmatrix} 0&0\\0&1 \end{bmatrix}$, $\IWR$$(A,B)=4$.  
\end{proposition}
\begin{proof}
	We have $$M^3(A,B)=\begin{bmatrix}
	1&0&1&0&0&0&1&0\\0&0&0&0&0&1&0&0\\1&0&0&0&1&0&1&0\\0&0&0&1&0&0&0&0\\0&0&1&0&0&0&1&0\\0&1&0&0&0&1&0&0\\1&0&1&0&1&0&1&0\\0&0&0&0&0&0&0&0
	\end{bmatrix}.$$ Reordering columns of $M^3(A,B)$ in order given by the permutation $51732648$ yields the matrix $\begin{bmatrix}
	0&1&1&1&0&0&0&0\\
	0&0&0&0&0&1&0&0\\
	1&1&1&0&0&0&0&0\\
	0&0&0&0&0&0&1&0\\
	0&0&1&1&0&0&0&0\\
	0&0&0&0&1&1&0&0\\
	1&1&1&1&0&0&0&0\\
	0&0&0&0&0&0&0&0
	\end{bmatrix}$.\\[2mm] By Theorem \ref{Thm-permuterow010}, we have $G^3(A,B)$ is word-representable. For 
	
	$M^4(A,B) = \begin{bmatrix}
	0&0&1&0&0&0&1&0&1&0&1&0&0&0&1&0\\
	0&1&0&0&0&1&0&0&0&0&0&0&0&1&0&0\\
	1&0&1&0&1&0&1&0&1&0&0&0&1&0&1&0\\
	0&0&0&0&0&0&0&0&0&0&0&1&0&0&0&0\\
	0&0&1&0&1&0&1&0&0&0&1&0&0&0&1&0\\
	0&1&0&0&0&0&0&0&0&1&0&0&0&1&0&0\\
	1&0&1&0&1&0&0&0&1&0&1&0&1&0&1&0\\
	0&0&0&0&0&0&0&1&0&0&0&0&0&0&0&0\\
	1&0&1&0&0&0&1&0&1&0&1&0&0&0&1&0\\
	0&0&0&0&0&1&0&0&0&0&0&0&0&1&0&0\\
	1&0&0&0&1&0&1&0&1&0&0&0&1&0&1&0\\
	0&0&0&1&0&0&0&0&0&0&0&1&0&0&0&0\\
	0&0&1&0&0&0&1&0&0&0&1&0&0&0&1&0\\
	0&1&0&0&0&1&0&0&0&1&0&0&0&1&0&0\\
	1&0&1&0&1&0&1&0&1&0&1&0&1&0&1&0\\
	0&0&0&0&0&0&0&0&0&0&0&0&0&0&0&0
	\end{bmatrix},$ \\
	
\noindent	
suppose that a reordering of columns $\rho = \rho_1\rho_2 \cdots \rho_{16}$ exists showing word-representability of $G^4(A,B)$ by Theorem \ref{Thm-permuterow010}. Then,
	\begin{itemize}
		\item  from row 3, columns 1, 3, 5, 7, 9, 13 and 15 must be (cyclically) consecutive;
		\item  from row 5, columns 3, 5, 7, 11 and 15 must be (cyclically) consecutive;
		\item  from row 7, columns 1, 3, 5, 9, 11, 13 and 15 must be (cyclically) consecutive.
	\end{itemize}
But then, from the first and the third bullet points, columns 1, 3, 5, 9, 13 and 15 must be consecutive and then column 7 or 11 is next to the left of them and the other one is next to the right of them. This contradicts to the second bullet point. So there is no such $\rho$ and $G^4(A,B)$ is non-word-representable. Therefore, $\IWR$$(A,B) = 4$.
\end{proof}

\begin{remark}\label{REMindex=4}
Proposition~\ref{caseindex=4} gives Case 6 in Table~\ref{tab:2x2-1}. In each of Cases 53, 63, 83, 87, 93, 98 and 105 in Tables~\ref{tab:2x2-1} and~\ref{tab:2x2-2},  $M^3(A,B)$ 
has a permutations satisfying condition in Theorem~\ref{Thm-ABcolpermute} but $M^4(A,B)$ does not (similarly to Proposition~\ref{caseindex=4}). So $\IWR$$(A,B)=4$ for $A$ and $B$ in these cases. Moreover, by Theorem~\ref{Thm-ABcolpermute}, column and row permutations of $A$ and $B$ give the same $\IWR$. Consequently, we also have $\IWR$$(A,B)=3$ for $A$ and $B$ in Cases 21, 35, 54, 64, 85, 92 and 95.
	\end{remark}

By Remarks~\ref{REMindex=2}, \ref{REMindex=3} and \ref{REMindex=4}, we can see that in many cases the index of word-representability is 2, 3 or 4. Next, we will introduce certain definitions and theorems to present the cases where the index of word-representability is infinity.

Let $M$ be an $m \times n$ binary matrix. For convenience, we will represent rows of $M$ by binary strings of length $n$. For example, we will represent three rows of ${\tiny \begin{bmatrix} 1&1&0&1\\0&1&0&0\\0&0&0&1 \end{bmatrix}}$ by 1101, 0100 and 0001.

\begin{definition}
	Let $A$ and $B$ be $m \times n$ binary matrices. Define $R^k(A,B)$ to be the  set of binary strings representing rows of $M^k(A,B)$. So every element of $R^k(A,B)$ is a binary string of length $n^k$. Each element of $R^k(A,B)$ is called a {\em row pattern of $M^k(A,B)$}.
\end{definition}

\begin{definition} Let $A=\begin{bmatrix}a&b\\c&d\end{bmatrix}$ and $B=\begin{bmatrix}e&f\\g&h\end{bmatrix}$ be $2 \times 2$ binary matrices and $B^n$ be the set of binary strings of length $n$. We define functions $u_{A,B}:\{0,1\} \rightarrow B^2$ and $l_{A,B}:\{0,1\} \rightarrow B^2$ by $$u_{A,B}(0)=ab ,l_{A,B}(0)=cd, u_{A,B}(1)= ef  \text{~~and~~} l_{A,B}(1)=gh.$$
Moreover, if $v=v_1v_2 \cdots v_k \in B^k$, $k\geq 2$, we extend the definition of the functions $u_{A,B}$  and $l_{A,B}$ to the case of $B^k \rightarrow B^{2k}$  by $$u_{A,B}(v)=u_{A,B}(v_1)u_{A,B}(v_2) \cdots u_{A,B}(v_k) $$ and $$l_{A,B}(v)=l_{A,B}(v_1)l_{A,B}(v_2) \cdots l_{A,B}(v_k). $$
When $A$ and $B$ are clear from the context, we can omit the subscript and write $u$ and $l$ instead of $u_{A,B}$ and $l_{A,B}$, respectively.
\end{definition}

\begin{example} \label{Ex-ul}
	{\em
	Let $A= \begin{bmatrix}1&1\\0&0\end{bmatrix}$ and $B=\begin{bmatrix} 1&0\\1&0	\end{bmatrix}$.  Then we have $$M^3(A,B)= \begin{bmatrix}
	1&0&1&1&1&0&1&1\\
	1&0&0&0&1&0&0&0\\
	1&0&1&1&1&0&1&1\\
	1&0&0&0&1&0&0&0\\
	1&0&1&0&1&0&1&0\\
	1&0&1&0&1&0&1&0\\
	1&1&1&1&1&1&1&1\\
	0&0&0&0&0&0&0&0
	\end{bmatrix}$$
	and $R^4(A,B)= \{ 10111011,10001000,10101010,11111111,00000000 \}$. In fact, we can find $R^4(A,B)$ by using the functions $u_{(A,B)}$ and $l_{A,B}$. As we start with $M^0(A,B)=\begin{bmatrix}0\end{bmatrix}$, we have $R^0(A,B)=\{0\}$. Then we apply the functions $u_{A,B}$ and $l_{A,B}$ to all elements in $R^0(A,B)$ to get all elements in $R^1(A,B)$:
	\begin{align*}
	u_{A,B}(0) = 11 \text{~~and~~} l_{A,B} (0) = 00.
	\end{align*}
	So, $R^1(A,B)=\{11,00\}$. Now,
	\begin{align*}
	u_{A,B}(11) = 1010 &\text{~~and~~} l_{A,B}(11) = 1010 \\
	u_{A,B}(00) = 1111 &\text{~~and~~} l_{A,B}(00) = 0000
	\end{align*}
	so $R^2(A,B)=\{1010,1111,0000\}$. Repeating the procedure one more time yields 
	\begin{align*}
	u_{A,B}(1010) = 10111011 &\text{~~and~~} l_{A,B}(1010) = 10001000 \\
	u_{A,B}(1111) = 10101010 &\text{~~and~~} l_{A,B}(1111) = 10101010 \\
	u_{A,B}(0000) = 11111111 &\text{~~and~~} l_{A,B}(0000) = 00000000,
	\end{align*}
	and so $R^3(A,B)= \{10111011,10001000,10101010,11111111,00000000\}$ which is the same as $R^4(A,B)$.
	}
\end{example} 

So, all elements in $R^k(A,B)$ are obtained by applying $u_{A,B}$ and $l_{A,B}$ to every element in $R^{k-1}(A,B)$. The next theorem generalizes this observation, and it can be proved easily by induction.

\begin{theorem}
	Let $A$ and $B$ be $2 \times 2$ binary matrices. Then $$R^k(A,B)= \{ f_k(\cdots f_2(f_1(0)) \cdots )| f_i \in \{ u_{A,B}, l_{A,B}\} \} \text{~for~any~} k \geq 1.$$ 
\end{theorem}

\begin{definition}
		Let $v=v_1v_2 \cdots v_k\in B^k$. Then, $\Gamma(v) := \{m\in \{1,2, \ldots, k \} | v_m =1 \}$.
\end{definition}

In order to study row patterns, we introduce a relation $\leq$ on $B^n$. Let $x=x_1x_2\cdots x_k$ and $y=y_1y_2\cdots y_k$ be in $B^k$. We say that $x \leq y$ if and only if $x_i=1$ implies $y_i = 1$ for every $i\in \{1,2,\ldots,k\}$. In other words,  $x \leq y$ if and only if $\Gamma(x) \subseteq \Gamma(y)$. It is easy see that $\leq$ is reflexive, antisymmetric and transitive, and thus $\leq$ is a partial order. 

\begin{theorem} \label{ThmTO->WR}
	Let $A$ and $B$ be $m \times n$ binary matrices. For any $k>1$, if $(R^k(A,B),\leq)$ is a total order, then $G^k(A,B)$ is word-representable.
\end{theorem}
\begin{proof}
	Let $R^k(A,B)= \{x_1,x_2,\ldots, x_l\}$ where $l\geq 1$ and  $x_1,x_2,\ldots, x_l$ are binary strings of length $n^k$. Since $(R^k(A,B),\leq)$ is a total order, w.l.o.g., we assume that $x_1 \leq x_2 \leq \cdots \leq x_l$. That is $\Gamma(x_1) \subseteq \Gamma(x_2) \subseteq \cdots \subseteq \Gamma(x_l)$. Let
	\begin{align*}
	D_1 &:= \Gamma(x_1),\\
	D_2 &:= \Gamma(x_2)\setminus \Gamma(x_1),\\
	D_3 &:= \Gamma(x_3) \setminus \Gamma(x_2),\\
		& \text{~~}\vdots \\
	D_l &:= \Gamma(x_l) \setminus \Gamma(x_{l-1}) \text{~~~~and}\\
	D_{l+1} &:= \{1,2, \ldots ,n^k\} \setminus \Gamma(x_l).
	\end{align*}
If $D_j=\{i_{j,1},i_{j,2},\ldots,i_{j,|D_j|}\}$ for $i_{j,1}<i_{j,2}<\cdots <i_{j,|D_j|}$, then $$\rho = i_{1,1}i_{1,2}\cdots i_{1,|D_1|}i_{2,1}i_{2,2}\cdots i_{2,|D_2|}\cdots i_{l,1}i_{l,2}\cdots i_{l,|D_l|}$$ is a $n^k$-permutation. Let $M^*$ be the matrix obtained by reordering columns of $M^k(A,B)$ according to the order given by $\rho$. Then we want to prove that every row of $M^*(A,B)$ is of the form $1^s0^t$ for some $s,t \geq 0$. Let $y=y_1y_2\cdots y_{n^k}$ be a row pattern of $M^k(A,B)$ and $y^*=y_1^*y_2^* \cdots y_{n^k}^*$ be the row pattern after reordering $y$. Since $y \in R^k(A,B)$, then $y=x_q$ for some $q \in \{1,2, \ldots, l\}$. So $y_i=1$ for all $i \in \Gamma(x_q)$ and $y_i=0$ for all $i \notin \Gamma(x_q)$. That is, $\Gamma(y)$ is $$\{ i_{1,1},i_{1,2},\ldots,i_{1,|D_1|},i_{2,1},i_{2,2},\ldots,i_{2,|D_2|},\ldots, i_{q,1},i_{q,2},\ldots,i_{l,|D_q|} \}.$$ 
	Hence $y^*=1^s0^t$ where $s=|D_1|+|D_2|+\cdots+|D_q|$ and $t= n^k-|D_1|-|D_2|-\cdots-|D_q| $. Therefore every row of $M^*(A,B)$ is of the form $1^s0^t$. By Theorem~\ref{Thm-permuterow010}, we have that $G^k(A,B)$ is word-representable.
\end{proof}

 \begin{theorem} \label{ThmXYcompare->RW}
	Let $A$ and $B$ be  $2 \times 2$ binary matrices. If $x \leq y$ implies $\{u_{A,B}(x), l_{A,B}(x), u_{A,B}(y),l_{A,B}(y)\}$ is comparable under $\leq$ for binary strings  $x$ and $y$ of the same length, then $G^k(A,B)$ is word-representable for any $k \geq 0$.
\end{theorem}

\begin{proof}
	We will prove by induction that $R^k(A,B)$ is comparable under $\leq$. Because $0 \leq 0$, we have $\{ u_{A,B}(0),l_{A,B}(0) \}= R^1(A,B)$ is comparable under $\leq$. Suppose that $R^l(A,B)$ is comparable under $\leq$ for some $l \geq 1$. Let $v,w \in R^{l+1}(A,B)$. Then 
	\begin{align*}
	v=u_{A,B}(x) \text{~~or~~} v=l_{A,B}(x) \text{~~for~some~~} x \in R^l(A,B)
	\end{align*}
	and 
	\begin{align*}
	w=u_{A,B}(y) \text{~~or~~} w=l_{A,B}(y) \text{~~for~some~~} y \in R^l(A,B).
	\end{align*}
	By induction hypothesis, w.l.o.g., we can assume that $x \leq y$. So, $\{u_{A,B}(x), l_{A,B}(x), u_{A,B}(y),l_{A,B}(y)\}$ is comparable and $v$ and $w$ belong to this set. Thus we have that $v$ and $w$ are comparable. Hence $R^k(A,B)$ is comparable under $\leq$ for any $k \geq 0$. Hence, by Theorem~\ref{ThmTO->WR}, $G^k(A,B)$ is word-representable for any $k \geq 0$.
\end{proof}

\begin{theorem} \label{Thm100101->WR}
	Let $A$ and $B$ be  $2 \times 2$ binary matrices. If $\{ u_{A,B}(100),$ $ l_{A,B}(100), u_{A,B}(101),l_{A,B}(101)\}$ is comparable under $\leq$, then $G^k(A,B)$ is word-representable for any $k \geq 0$.
\end{theorem}

\begin{proof}
	For convenience, we write $u$ and $l$ instead of $u_{A,B}$ and $l_{A,B}$, respectively. Suppose that $\{u(100),l(100),u(101),l(101)\}$ is comparable under $\leq$. Let $x=x_1x_2\cdots x_m$ and $y=y_1y_2 \cdots y_m$ be binary strings of length $m$ such that $x \leq y$. Also, let 
	\begin{align*}
	R &:= \{i | x_i=0 ,y_i=0\},\\
	S &:= \{i | x_i=1 ,y_i=1\} \text{~and}\\
	T &:= \{i | x_i=0 ,y_i=1\}.
	\end{align*}
	Note that $R,S,T$ partition the set $\{1,2,\ldots,m\}$. Since $u(100)$ and $l(100)$ are comparable, we have two cases to consider.
	
	If $u(100) \leq l(100)$, then $u(1) \leq l(1)$ and $u(0) \leq l(0)$. So we have $u(101) \leq l(101)$ and it is impossible that $l(101) \leq u(100)$ and $l(100) \leq u(101)$. So there are four possible cases here, which are
	\begin{align*}
	&u(100) \leq u(101) \leq l(100) \leq l(101),\\
	&u(100) \leq u(101) \leq l(101) \leq l(100),\\
	&u(101) \leq u(100) \leq l(101) \leq l(100) \text{~and}\\
	&u(101) \leq u(100) \leq l(100) \leq l(101).
	\end{align*}
	Similarly, in the case of $l(100) \leq u(100)$, we have $l(1) \leq u(1)$ and $l(0) \leq u(0)$. So $l(101) \leq u(101)$ and  $u(100) \nleq l(100)$ and $u(101) \nleq l(100)$. So we have four more cases, which are
	\begin{align*}
	&l(100) \leq u(100) \leq l(101) \leq u(101),\\
	&l(100) \leq l(101) \leq u(101) \leq u(100),\\
	&l(101) \leq l(100) \leq u(101) \leq u(100) \text{~and}\\
	&l(101) \leq l(100) \leq u(100) \leq u(101).
	\end{align*}
Next, we will consider comparability of the set $\{u(x), l(x), u(y),$ $ l(y)\}$ in each case.
	
\begin{itemize}
\item $u(100) \leq u(101) \leq l(100) \leq l(101)$. So we have  $u(0) \leq u(1) \leq l(0) \leq l(1)$. Note that $u(x_i) \leq u(y_i)$ and $l(x_i) \leq l(y_i)$ for any $i \in T$. Then $u(x) \leq u(y)$ and $l(x) \leq l(y)$. Since $u(y_i) \leq l(x_i)$ where $i$ belongs to R, S or T, so $u(y) \leq l(x)$. Hence, $u(x) \leq u(y) \leq l(x) \leq l(y)$.

\item $u(100) \leq u(101) \leq l(101) \leq l(100)$. So we have  $u(0) \leq u(1) \leq l(1) \leq l(0)$. Note that $u(x_i) \leq u(y_i)$ and $l(y_i) \leq l(x_i)$ for any $i \in T$. Then $u(x) \leq u(y)$ and $l(y) \leq l(x)$. Since $u(y_i) \leq l(y_i)$ where $i$ belongs to R, S or T, so $u(y) \leq l(y)$. Hence, $u(x) \leq u(y) \leq l(y) \leq l(x)$.
	
\item $u(101) \leq u(100) \leq l(101) \leq l(100)$. So we have  $u(1) \leq u(0) \leq l(1) \leq l(0)$. Note that $u(y_i) \leq u(x_i)$ and $l(y_i) \leq l(x_i)$ for any $i \in T$. Then $u(y) \leq u(x)$ and $l(y) \leq l(x)$. Since $u(x_i) \leq l(y_i)$ where $i$ belongs to R, S or T, so $u(x) \leq l(y)$. Hence, $u(y) \leq u(x) \leq l(y) \leq l(x)$.

\item $u(101) \leq u(100) \leq l(100) \leq l(101)$. So we have  $u(1) \leq u(0) \leq l(0) \leq l(1)$. Note that $u(y_i) \leq u(x_i)$ and $l(x_i) \leq l(y_i)$ for any $i \in T$. Then $u(y) \leq u(x)$ and $l(x) \leq l(y)$. Since $u(x_i) \leq l(x_i)$ where $i$ belongs to R, S or T, so $u(x) \leq l(x)$. Hence, $u(y) \leq u(x) \leq l(x) \leq l(y)$.	
	
\item $l(100) \leq l(101) \leq u(100) \leq u(101)$. So we have  $l(0) \leq l(1) \leq u(0) \leq u(1)$. Note that $l(x_i) \leq l(y_i)$ and $u(x_i) \leq u(y_i)$ for any $i \in T$. Then $l(x) \leq l(y)$ and $u(x) \leq u(y)$. Since $l(y_i) \leq u(x_i)$ where $i$ belongs to R, S or T, so $l(y) \leq u(x)$. Hence, $l(x) \leq l(y) \leq u(x) \leq u(y)$.
	
\item $l(100) \leq l(101) \leq u(101) \leq u(100)$. So we have  $l(0) \leq l(1) \leq u(1) \leq u(0)$. Note that $l(x_i) \leq l(y_i)$ and $u(y_i) \leq u(x_i)$ for any $i \in T$. Then $l(x) \leq l(y)$ and $u(y) \leq u(x)$. Since $l(y_i) \leq u(y_i)$ where $i$ belongs to R, S or T, so $l(y) \leq u(y)$. Hence, $l(x) \leq l(y) \leq u(y) \leq u(x)$.
	
\item $l(101) \leq l(100) \leq u(101) \leq u(100)$. So we have  $l(1) \leq l(0) \leq u(1) \leq u(0)$. Note that $l(y_i) \leq l(x_i)$ and $u(y_i) \leq u(x_i)$ for any $i \in T$. Then $l(y) \leq l(x)$ and $u(y) \leq u(x)$. Since $l(x_i) \leq u(y_i)$ where $i$ belongs to R, S or T, so $l(x) \leq u(y)$. Hence, $l(y) \leq l(x) \leq u(y) \leq u(x)$.
	
\item $l(101) \leq l(100) \leq u(100) \leq u(101)$. So we have  $l(1) \leq l(0) \leq u(0) \leq u(1)$. Note that $l(y_i) \leq l(x_i)$ and $u(x_i) \leq u(y_i)$ for any $i \in T$. Then $l(y) \leq l(x)$ and $u(x) \leq u(y)$. Since $l(x_i) \leq u(x_i)$ where $i$ belongs to R, S or T, so $l(x) \leq u(x)$. Hence, $l(y) \leq l(x) \leq u(x) \leq u(y)$.
\end{itemize}
	
	We can see that $\{u(x), l(x), u(y),$ $ l(y)\}$ is comparable in every case. By Theorem~\ref{ThmXYcompare->RW}, $G^k(A,B)$ is word-representable for any $k \geq 0$.
\end{proof}

\begin{remark} \label{Remindex=inf}
	Theorem~\ref{Thm100101->WR} can be applied to check word-representability of $G^k(A,B)$. We can see that $\{ u_{A,B}(100), l_{A,B}(100), u_{A,B}(101),l_{A,B}(101)\}$ is comparable under $\leq$ in Cases 2, 7, 8, 13, 51, 56, 61, 84, 94, 100, 103, 107 and 113 in Tables~\ref{tab:2x2-1} and~\ref{tab:2x2-2}. Then $\IWR$$(A,B)= \infty$ for $A$ and $B$ in these cases. By Theorem~\ref{Thm-ABcolpermute}, column and row permutations of $A$ and $B$ in these cases preserve the index of word-representability. So we also have $\IWR$$(A,B)=\infty$ for $A$ and $B$ in Cases 18, 23, 27, 30, 34, 43, 44, 48, 52, 59, 62, 86 and 96.
\end{remark}

\begin{proposition} \label{case9}
	Let $A=\begin{bmatrix} 1&0\\0&0 \end{bmatrix}$ and $B=\begin{bmatrix} 1&0\\0&1 \end{bmatrix}$. Then $G^k(A,B)$ is word-representable for any $k \geq 0$.
\end{proposition}
\begin{proof}
	The case of $k=0$ is trivial. For $k \geq 1$, we let $S^k_0 := \emptyset$,
	$$ S^k_j :=\{s\in \{1,2,\ldots,2^k\} \vert s \equiv 2^{j-1} \pmod{2^j} \} \text{~for~} j \in \{1,2,\ldots , k\} $$ and 
	\begin{align*}
	T^k := \{ x_1x_2\cdots x_{2^k} \vert x_i = \begin{cases}
		0 \text{~if~} i \notin S^k_j,\\ 
		1 \text{~if~} i \in S^k_j 
	\end{cases} 
	\text{~for~some~} j \in \{0,1,\ldots, 2^k\}	\}. 
	\end{align*}	
	We claim that $R^k(A,B) = T^k$ for any $k \geq 1$ and prove it by induction on $k$. It is obvious in the case of $k=1$ because $R^1(A,B)=\{00,10\}$ while $S_0=\emptyset$ and $S_1 = \{1\}$. Suppose $l$ is a positive integer such that $R^l(A,B) = T^l$. Let $y \in R^{l+1}(A,B)$, then $y=u_{A,B}(z)$ or $y=l_{A,B}(z)$ for some $z \in R^l(A,B)$. Since $R^l(A,B)=T^l$, we have $$z=z_1z_2\cdots z_{2^l} \text{~where~} z_i= \begin{cases}
	0 \text{~if~} i \notin S^l_j,\\ 
	1 \text{~if~} i \in S^l_j. 
	\end{cases}$$
	If $y=u_{A,B}(z)$, then $y=1010\cdots10$. That is, 
	$y_i = \begin{cases}
	0 \text{~if~} i \notin S^{l+1}_1,\\ 
	1 \text{~if~} i \in S^{l+1}_1 
	\end{cases}$, and so $y \in T^{l+1}$. If $y=l_{A,B}(z)$, the number of $1$'s in $y$ and $z$ is identical because $0$ is mapped to $00$ and $1$ is mapped to $01$. We can see that $y_{2i}=1$ if and only if $x_i=1$. Hence, $$y_i = \begin{cases}
	0 \text{~if~} i \notin S^{l+1}_{r+1},\\ 
	1 \text{~if~} i \in S^{l+1}_{r+1}. 
	\end{cases}$$ So $R^{l+1}(A,B) \subseteq T^{l+1}$. Conversely, let $v=v_1v_2\cdots v_{2^{l+1}}$ where $$v_i = \begin{cases}
	0 \text{~if~} i \notin S^{l+1}_{t},\\ 
	1 \text{~if~} i \in S^{l+1}_{t} 
	\end{cases}$$ for some $t \in \{0,1,\ldots, l+1 \}$. If $t=0$, note that $v=000\cdots0 = l_{A,B}(000\cdots0)$ and $000\cdots0 \in R^l(A,B)$, and so $v \in R^{l+1}(A,B)$. If $t=1$, we have $v=1010 \cdots 10 = u_{A,B}(\bar{v})$ for any $\bar{v} \in R^l(A,B)$. That is $v \in R^{l+1}(A,B)$. Suppose $t \in \{2,3,\cdots l+1\}$ and $w=w_1w_2 \cdots w_{2^l}$ be an element of $T^l$ such that $$w_i = \begin{cases}
	0 \text{~if~} i \notin S^l_{t-1},\\ 
	1 \text{~if~} i \in S^l_{t-1}. 
	\end{cases}$$ By induction hypothesis, $w \in R^l(A,B)$. Note that $v= l_{A,B}(w)$ and then $v \in R^{l+1}(A,B)$. So we have $T^{l+1} \subseteq R^{l+1}(A,B)$. Hence we have already proved the claim.
	
	Note that $\{1,2,\ldots, 2^k-1 \}$ is a disjoint union of $S_1,S_2,\ldots, S_k$. If $S^k_j = \{i_{j,1},i_{j,2},\ldots , i_{j,|S^k_j|}\}$ for $i_{j,1} < i_{l,2} < \cdots < i_{j,|S^k_j|}$, then we can let $\rho$ be the permutation 
	$$i_{1,1}i_{1,2}\cdots i_{1,|S^k_1|}
	i_{2,1}i_{2,2}\cdots i_{2,|S^k_2|} \cdots
	i_{k,1}i_{k,2}\cdots i_{k,|S^k_k|}2^k.
	$$
	Since $R^l(A,B)=Tt^k$, then all $1$'s in each row pattern in $R^k(A,B)$ is in columns $i_{j,1},i_{j,2},\ldots , i_{1,|S^k_j|}$ for some $1<j<k$. Therefore we can see that every row of the matrix obtained by reordering columns of $M^k(A,B)$ according to the order given by $\rho$ is of the form $0^a1^b0^c$ for some non-negative integers $a,b$ and $c$. By Theorem~\ref{Thm-permuterow010}, $G^k(A,B)$ is word-representable for ant $k \geq 1$.
\end{proof}

	\begin{remark} 
	\label{case9-26-41}
	Proposition~\ref{case9} gives $\IWR$$(A,B)= \infty$ for $A$ and $B$ in Case 9 in Tables~\ref{tab:2x2-1}. By Theorem~\ref{Thm-ABcolpermute}, column and row permutations of $A$ and $B$ give the same $\IWR$. Consequently, we also have $\IWR$$(A,B)=\infty$ for $A$ and $B$ in Cases 26 and 41.
	\end{remark}

	\begin{proposition} \label{case74}
		Let $A=\begin{bmatrix} 1&0\\0&1 \end{bmatrix}$ and $B=\begin{bmatrix} 0&1\\1&0 \end{bmatrix}$. Then $G^k(A,B)$ is word-representable for any $k \geq 0$.
	\end{proposition}
	\begin{proof}
	The case of $k=0$ is trivial. For any binary string $x$, we define $\bar{x}$ to be the binary string obtained by changing digits of $x$ from 0 to 1 and from 1 to 0. We claim that, for $k \geq 1$, $R^k(A,B)= \{x, \bar{x}\}$ for some $x \in B^{2^k}$ and prove it by induction on $k$. As $R^1(A,B)= \{10,01\}$ and $01 = \overline{10}$, the case of $k=1$ is done. Suppose that $k>1$ is a positive integer such that $R^k(A,B)=\{ x, \bar{x} \}$ for some $x \in B^{2^k}$. Then $R^{k+1}(A,B)=\{ u_{A,B}(x),l_{A,B}(x),u_{A,B}(\bar{x}), l_{A,B}(\bar{x})\}$. Note that $u_{A,B}(0)=l_{A,B}(1)$ and $u_{A,B}(1)=l_{A,B}(0)$. Then $u_{A,B}(x)=l_{A,B}(\bar{x})$ and $l_{A,B}(x)=u_{A,B}(\bar{x})$. So we have $R^k(A,B)=\{ u_{A,B}(x) , l_{A,B}(x) \}$. Since $u_{A,B}(0)=\overline{l_{A,B}(0)}$ and $u_{A,B}(1)=\overline{l_{A,B}(1)}$, we have $u_{A,B}(x)= \overline{l_{A,B}(x)}$. That is, $R^k(A,B)=\{u_{A,B}(x), \overline{u_{A,B}(x)}\}$. Hence we have proved the claim by induction. 
	
So, for any $k\geq1$, $R^k(A,B)=\{x,\overline{x}\}$ for some binary string $x$. Suppose that there are $s$ $1$'s in $x$. Let $\rho$ be a permutation such that reordering $\begin{bmatrix}x\end{bmatrix}$ according to the order giving by $\rho$ make all $1$'s in $x$ be together in the first $s$ columns. Then reordering $\begin{bmatrix}\bar{x}\end{bmatrix}$ according to the order giving by $\rho$ makes all $1$'s in $x$ be together in the last $2^k-s$ columns. Hence each row of the matrix obtained by reordering $M^k(A,B)$ according to the order giving by $\rho$ is $1^s0^{2^k-s}$ or $0^s1^{2^k-s}$. By Theorem~\ref{Thm-permuterow010}, $G^k(A,B)$ is word-representable for any $k \geq 0$.
	\end{proof}

\subsection{The case when A is the all-one matrix}

There are only 16 cases when $A= \begin{bmatrix} 1&1\\1&1 \end{bmatrix}$. If $B= \begin{bmatrix}1&1\\1&1\end{bmatrix}$, we have $M^0(A,B)=\begin{bmatrix}0\end{bmatrix}$ and $M^k(A,B)$ is all-one matrix for any $k\geq1$ which is  word-representable by Theorem~\ref{Thm-permuterow010}. The rest of this paper, except Theorem~\ref{Thm-2chains}, deals with the cases when $B$ has at least one 0.

\begin{theorem} \label{Thm-2chains}
 	Let  $A=\begin{bmatrix} 1	\end{bmatrix}_{m \times n}$ and $B$ be an $m \times n$ binary matrix. Then, $M^k(A,B)$ is a submatrix of $M^{k+2}(A,B)$, and $G^k(A,B)$ is an induced subgraph of $G^{k+2}(A,B)$, for any $k \geq 0$.
\end{theorem}
\begin{proof}
Since the case of	$B$ being an all-one matrix is trivial, we assume that $B$ is not an all-one matrix. Let $B=\begin{Bmatrix} b_{ij} \end{Bmatrix}_{m \times n}$ and $b_{rs}=0$ for some $1 \leq r \leq m$ and $1 \leq s \leq n$.  We will prove by induction on $k$ that $M^k(A,B)$ is contained in $M^{k+2}(A,B)$ as a submatrix by rows $(r-1)m^k+1, (r-1)m^k+2,\ldots, rm^k$ and columns $(s-1)n^k+1, (s-1)n^k+2,\ldots, sn^k$ for any $k \geq 0$.

Note that $M^0(A,B)= \begin{bmatrix}0\end{bmatrix}$, $M^1(A,B)= \begin{bmatrix} 1 \end{bmatrix}_{m \times n}$ and $$M^2(A,B)= \begin{bmatrix}
	B&B&\cdots & B\\
	B&B&\cdots & B\\
	\vdots&\vdots&\ddots&\vdots \\
	B&B&\cdots & B\\
	\end{bmatrix}.$$ Let $M^2(A,B)=\begin{bmatrix} m_{ij} \end{bmatrix}$, then $m_{rs}=0$. So $M^0(A,B)$ is a  submatrix of $M^2(A,B)$ by row $(r-1)m^0+1$ and column $(s-1)n^0+1$.
	
	Let $l\geq 1$ be an integer such that $M^l(A,B)$ is a  submatrix of $M^{l+2}(A,B)$ by rows $(r-1)m^k+1, (r-1)m^k+2,\ldots, rm^k$ and columns $(s-1)n^k+1, (s-1)n^k+2,\ldots, sn^k$. For the next iteration of morphism applied to $M^{l+2}(A,B)$, it is easy to see that $M^l(A,B)$  in the rows $(r-1)m^k+1, (r-1)m^k+2,\ldots, rm^k$ and the columns $(s-1)n^k+1, (s-1)n^k+2,\ldots, sn^k$ $M^{l+2}(A,B)$ is mapped to $M^{l+1}(A,B)$ in the rows $(r-1)m^{k+1}+1, (r-1)m^{k+1}+2,\ldots, rm^{k+1}$ and columns $(s-1)n^{k+1}+1, (s-1)n^{k+1}+2,\ldots, sn^{k+1}$ of $M^{l+3}(A,B)$. Hence, $M^k(A,B)$ is a submatrix of $M^{k+2}(A,B)$ for any $k \geq 0$. Consequently, $G^k(A,B)$ is an induced subgraph of $G^{k+2}(A,B)$ for any $k \geq 0$.
\end{proof}

We know from Theorem~\ref{Thm-2chains} that $G^0(A,B) \preceq G^2(A,B) \preceq G^4(A,B) \preceq \cdots $ and $G^1(A,B) \preceq G^3(A,B) \preceq G^5(A,B) \preceq \cdots $ where $G \preceq H$ means $G$ is an induced subgraph of $H$. So we are interested in investigating the smallest integer $l$ such that $G^l(A,B)$ is non-word-representable in both cases of $l$ being  even and $l$ being odd. The cases 114, 119, 120, 123, 124 and 129 in Table~\ref{tab:2x2-2} are given by using Proposition~\ref{ABall0all1}. Theorem~\ref{Thm100101->WR} can be applied to Cases 125, 126, 127 and 128, and the index of word-representability in these cases is infinity. The following propositions discuss the remaining cases.

\begin{proposition}\label{Propcase121}
	For $A=\begin{bmatrix} 1&1\\1&1 \end{bmatrix}$ and $B=\begin{bmatrix} 1&0\\0&1 \end{bmatrix}$, $\IWR$$(A,B)=3$. Moreover, $G^{k}(A,B)$ is not word-representable for $k\geq 3$.
\end{proposition}

\begin{proof}
	Note that 
	$M^2(A,B)=\begin{bmatrix}1&0&1&0\\0&1&0&1\\1&0&1&0\\0&1&0&1\end{bmatrix}$ and $$M^3(A,B)= \begin{bmatrix}
	1&0&1&1&1&0&1&1\\
	0&1&1&1&0&1&1&1\\
	1&1&1&0&1&1&1&0\\
	1&1&0&1&1&1&0&1\\
	1&0&1&1&1&0&1&1\\
	0&1&1&1&0&1&1&1\\
	1&1&1&0&1&1&1&0\\
	1&1&0&1&1&1&0&1
	\end{bmatrix}.  $$
	Reordering columns of $M^2(A,B)$ in order given by 1324 yields the matrix $\begin{bmatrix}
	1&1&0&0\\0&0&1&1\\1&1&0&0\\0&0&1&1
	\end{bmatrix}$. By Theorem \ref{Thm-permuterow010}, $G^2(A,B)$ is word-representable. Let $M^*$ be a matrix obtained from reordering columns of $M^3(A,B)$ and every rows of $M^*$ are of the form $0^r1^s0^t$ or $0^r1^s0^t$. Since every rows of $M^3(A,B)$ has 2 zero entries, there is a row of $M^*$ of the form $1^a0^21^{6-a}$ where $ 1 \leq a \leq 5$. Note that there exist a row of $M^*$ having $1$'s in the columns $a$ and $a+2+1$. So $M^*$ cannot satisfy the second condition of Theorem~\ref{Thm-genpermutecolumn}. Hence $G^3(A,B)$ is non-word-representable, and so $\IWR$$(A,B)=3$.
	
	Now we consider the word-representability of $G^4(A,B)$. We have  
	\begin{align*}
	R^4(A,B) = \{ &1011101010111010, 0111010101110101, 1110101011101010,\\  &1101010111010101, 1010101110101011, 0101011101010111,\\ &1010111010101110, 0101110101011101 \}.
	\end{align*}
In order to apply Theorem~\ref{Thm-genpermutecolumn}, we assume the existence of an order $\rho = \rho_1\rho_2 \cdots \rho_{16}$ with the following properties:
	\begin{itemize}
		\item  from $1011101010111010 \in R^4(A,B)$, columns 2, 6, 8, 10, 14 and 16 must be cyclically consecutive;
		\item  from $1110101011101010 \in R^4(A,B)$, columns 4, 6, 8, 12, 14 and 16 must be cyclically consecutive;
		\item  from $1010101110101011 \in R^4(A,B)$, columns 2, 4, 6, 10, 12 and 14 must be cyclically consecutive.
	\end{itemize}
	It follows from the first and the second bullet points that 6, 8, 14, and 16 must be consecutive and then, w.l.o.g., 2 and 10 is next to the left of them and then 4 and 12 is next to the right them That means that 2 and 4 cannot be cyclically consecutive, which contradicts the second bullet point. So there is no such $\rho$ and $G^4(A,B)$ is non-word-representable. As $G^3(A,B)$ is non-word-representable, and the class of word-representable graphs is hereditary, by Theorem~\ref{Thm-2chains}, $G^{2k+1}(A,B)$ is non-word-representable for any $k \geq 1$. Similarly, $G^{2k}(A,B)$ is non-word-representable for any $k \geq 2$ because $G^4(A,B)$ is non-word-representable. Therefore, $G^{k}(A,B)$ is not word-representable for $k\geq 3$.\end{proof}

\begin{remark} \label{REMcase122}
	Proposition~\ref{Propcase121} gives Case 121 in Tables~\ref{tab:2x2-1}, and we obtain $\IWR$$(A,B)=3$ in Case 122 by a column and row permutation of $A$ and $B$.	
\end{remark}

\begin{proposition}\label{Propcase117}
	For $A=\begin{bmatrix} 1&1\\1&1 \end{bmatrix}$ and $B=\begin{bmatrix} 0&0\\1&0 \end{bmatrix}$, $\IWR$$(A,B)=5$.  Moreover, $G^{k}(A,B)$ is not word-representable for $k\geq 5$. 
\end{proposition}

\begin{proof} 
	It is easy to see that $G^0(A,B)$, $G^1(A,B)$ and $G^2(A,B)$ are word-representable. Since reordering columns of $M^3(A,B)$ in order given by the permutation 26153748 yields a matrix satisfying the condition in Theorem~~\ref{Thm-permuterow010}, $G^3(A,B)$ is word-representable. Further, reordering columns of $M^4(A,B)$ in order given by the permutation $$2(10)4(12)3(11)19 5(13)7(15)6(14)(16)$$ also yields a matrix satisfying the condition in Theorem~\ref{Thm-permuterow010}, so $G^4(A,B)$ is word-representable. 
	
Next, we  consider $G^5(A,B)$. We use  $u$ and $l$ instead of $u_{A,B}$ and $l_{A,B}$, respectively.
In order to apply Theorem~\ref{Thm-genpermutecolumn}, we assume the existence of an order $\rho = \rho_1\rho_2 \cdots \rho_{32}$ proving word-representability of $G^5(A,B)$.
	\begin{itemize}
		\item $l(l(l(u(u(0)))))=10111011101110111011101110111011\in R^5(A,B)$ so columns 2, 6, 10, 14, 18, 22, 26 and 30 must be cyclically consecutive in $\rho$.
		\item $l(u(u(l(l(0)))))=10101010111111111010101011111111\in R^5(A,B)$ so columns 2, 4, 6, 8, 18, 20, 22 and 24 must be cyclically consecutive in $\rho$.		
		\item $l(u(l(l(u(0)))))=11111010111111111111101011111111\in R^5(A,B)$ so columns 6, 8, 22 and 24 must be cyclically consecutive in $\rho$.		
		\item $u(u(l(l(u(0)))))=11110000111111111111000011111111\in R^5(A,B)$ so columns 5, 6, 7, 8, 21, 22, 23 and 24 must be cyclically consecutive in $\rho$.
	\end{itemize}
	
	It follows from the first and the second bullet points that 2, 6, 18, and 22 must be consecutive and then, w.l.o.g., 10, 14, 26 and 30 are next to the left of these numbers and 4, 8, 20 and 24 are next to the right them. Hence $\rho$ contains $$\{ 10, 14, 26, 30 \}, \{ 2, 6, 18, 22 \}, \{ 4, 8, 20, 24 \}$$ where numbers in $\{ \}$ are consecutive in $\rho$ but are in some unknown to us order. From the third bullet point, $\rho$ contains $$\{ 10, 14, 26, 30 \}, \{ 2, 18 \}, \{ 6, 22 \}, \{ 8, 24 \}, \{ 4, 20 \}.$$ We obtain a contradiction with the fourth bullet point. So there is no such $\rho$ and $G^5(A,B)$ is non-word-representable. 
	
	Next, we consider word-representability of $G^6(A,B)$. In order to apply Theorem~\ref{Thm-genpermutecolumn}, we assume that there exist a permutation $\tau = \tau_1\tau_2 \cdots \tau_{64}$ proving word-representability of $G^6(A,B)$.
	\begin{itemize}
		\item $l(u(u(u(u(u(0)))))) \in R^6(A,B)$ is \\
		1010101010101010101010101010101010101010101010101010101010101010.\\ So, all odd columns must be cyclically consecutive in $\tau$.
		\item $u(u(u(u(l(l(0)))))) \in R^6(A,B)$ is \\
		111111111111111110000000000000000111111111111111110000000000000000.\\ So, columns in $\{1,2,\ldots,16,33,34,\ldots48\}$  must be cyclically consecutive in $\tau$.
		\item $u(l(l(u(u(u(0)))))) \in R^6(A,B)$ is \\
		0011000000110000001100000011000000110000001100000011000000110000.\\ So, columns in $\{3,4,11,12,19,20,27,28,35,36,43,44,51,52,59,60\}$  must be cyclically consecutive in $\tau$.
	\end{itemize}
	Since all odd columns must be cyclically consecutive in $\tau$, there is a unique $s \in \{1,2,\ldots,64\}$ such that $\tau_s$ is odd and $\tau_{s+1}$ is even, where for $s=64$, $s+1:=1$. From the second bullet point, we have
	\begin{align*}
	&  \{ \tau_{s},\tau_{s-1},\ldots,\tau_{s-15}\} = \{ 1,3,5,\ldots,15,33,35,37,\ldots,47 \}\\
	\text{~~and~~} & \{ \tau_{s+1},\tau_{s+2},\ldots,\tau_{s+16}\} = \{2,4,6,\ldots,16,34,36,38,\ldots,48\}
	\end{align*}
	where for $i\geq 0$, $\tau_{-i}:=\tau_{64-i}$ and for $s+i\geq 65$, $s+i:=s+i-64$. On the other hand, from the third bullet point, we have
	\begin{align*}
	&  \{ \tau_{s},\tau_{s-1},\ldots,\tau_{s-7}\} = \{ 3,11,19,27,35,43,51,59 \}\\
	\text{~~and~~} & \{ \tau_{s+1},\tau_{s+2},\ldots,\tau_{s+8}\} = \{4,12,20,28,36,44,52,60\}
	\end{align*}
	where the indices less than 1 and larger than 64 are treated as above. We obtain a contradiction because $ 19 \notin  \{ \tau_{s},\tau_{s-1},\ldots,\tau_{s-15}\} $ but $19 \in \{ \tau_{s},\tau_{s-1},\ldots,\tau_{s-7}\}$. Hence, there is no such $\tau$ and $G^6(A,B)$ is non-word-representable.
	Since the graphs $G^5(A,B)$ and $G^6(A,B)$ are not word-representable and the class of word-representable graphs is hereditary, by Theorem~\ref{Thm-2chains}, $G^k(A,B)$ is non-word-representable for any $k \geq 5$.
	\end{proof}

\begin{remark} \label{REMcase115-116-118}
Proposition~\ref{Propcase117} is Case 117 in Table~\ref{tab:2x2-1}. By Theorem~\ref{Thm-ABcolpermute}, column and row permutations of $A$ and $B$ give the same $\IWR$. Consequently, we also have $\IWR$$(A,B)=5$ for $A$ and $B$ in Cases 115, 116 and 117.	
\end{remark}

\section{Concluding remarks}\label{final-sec}

This paper is a major contribution to the study of word-representability of split graphs. Two key achievements in the paper are as follows:
\begin{itemize}
	\item Theorem~\ref{Thm-genpermutecolumn} can be used to study word-representability of split graphs via adjacency matrices, which is a novel approach.
	\item Necessary conditions for word-representability of split graphs obtained by iteration of morphism can be checked in polynomial time. Indeed, for a given such graph defined by matrices $A$ and $B$, we can go through all of $2 \times 2$ submatrices in $A$, and the respective submatrices in $B$, and then use our classification results in Tables~\ref{tab:2x2-1} and~\ref{tab:2x2-2} to detect non-word-representability.
\end{itemize}

As for open directions of research, it would be useful to provide a classification, similar to that we provided for $2 \times 2$ matrices in Tables~\ref{tab:2x2-1} and~\ref{tab:2x2-2}, for larger matrices, for example, for $2\times 3$ matrices, or $3 \times 3$ matrices, etc. This would enlarge our knowledge of word-representable split graphs obtained by iteration of morphisms, and hopefully, will eventually lead to a complete classification of such graphs.

For another research question, we noted in Tables~\ref{tab:2x2-1} and~\ref{tab:2x2-2} that if $G^4(A,B)$ is word-representable, then $IWR(A,B)=\infty$. In other words, the largest finite IWR in the case of $2 \times 2$ matrices is 4. Is there a reason for that? Does there exist a positive integer $t$ (a constant, or a function of $n$ and $m$) making the following statement true ``If $A,B$ are $n \times m$ matrices and $G^t(A,B)$ is word-representable, then $IWR(A,B)=\infty$"?

Finally, recall that if the leftmost bottom entry of $A$ is $1$ then $\lim_{k\to\infty}M^k(A,B)$ may not be well-defined as the sequence of graphs $G^k(A,B)$, for $k\geq 0$, may not be a chain of induced subgraphs. In all such cases, for $2 \times 2$ matrices, we still have that non-word-representability of $G^k(A,B)$ implies non-word-representability of $G^{k+1}(A,B)$. Is it always the case for $m\times n$ matrices $A$ and $B$? If not, then how do we characterize the situations when it is the case?

\end{document}